\documentclass[review,3p,times,sort&compress]{elsarticle}
\usepackage{amssymb}
\usepackage{tikz}
\usepackage{graphicx,fancyhdr}
\usepackage{graphicx}
\usepackage{subcaption}
\usepackage{multirow}
\usepackage{dsfont}
\usepackage{bbm}
\usepackage{booktabs}
\usepackage{amsthm}
\usepackage{bbm,bm}
\DeclareMathAlphabet{\mathpzc}{OT1}{pzc}{m}{it}
\usepackage{bbm}
\usepackage{dsfont}
\usepackage{yfonts}
\usepackage{amsmath}
\usepackage{amssymb}
\usepackage{amsfonts}
\usepackage{mathrsfs,float}
\usepackage{color}
\usepackage[colorlinks, citecolor=blue]{hyperref}

\makeatletter

\journal{}

\begin{document}
	
	\begin{frontmatter}
		
		
		\newtheorem{theorem}{Theorem}[section]
		\newtheorem{remark}[theorem]{Remark}
		\newtheorem{ex}{Example}
		\newtheorem{case}{Case}
		\newtheorem{pro}[theorem]{Proposition}
		\newtheorem{defi}[theorem]{Definition}
		\newtheorem{ass}{Assumption}
		\newtheorem{lemma}[theorem]{Lemma}
	\newtheorem{corollary}[theorem]{Corollary}
\newtheorem{assumption}[theorem]{Assumption}
		\newproof{pf}{Proof}
		\newproof{pot}{Proof of Theorem \ref{thm2}}
	\title{A Localized Core--Tail Fourier--Laguerre Frame Method with Adaptive Frequency Detection for Unbounded-Domain Problems}
			\author[1]{Chenyang Wang}
	\author[1]{Zhenyu Zhao}
\author[1]{Tinggang Zhao}
		\address[1]{School of Mathematics and Statistics, Shandong University of Technology, Zibo, 255049, China}

\begin{abstract}
We propose a localized core--tail Fourier--Laguerre frame method for
approximation and model problems on the real line. The domain is decomposed
into a finite core and two semi-infinite tails. Local Fourier extension is used
in the core to resolve nonperiodic, oscillatory, and locally nonsmooth
structures, while modulated Laguerre frames are used in the tails so that the
Laguerre functions approximate only slowly varying decaying envelopes. The
local Fourier extension component also provides two data-driven mechanisms:
coefficient-energy indicators for internal edge detection and local frequency
indicators for selecting tail modulation centers. We derive error estimates
that separate the core approximation error, the modulated Laguerre envelope
error, the effect of frequency mismatch, and the finite-tail truncation error.
The analysis shows that the tail complexity is governed mainly by the residual
phase after modulation rather than by the original carrier frequency. Numerical
experiments demonstrate high accuracy for oscillatory, multi-frequency, and
derivative-discontinuous functions, and a decaying model problem illustrates
that the representation can be combined with differential operators through
exact interface constraints.
\end{abstract}
		
	\end{frontmatter}

\section{Introduction}

Many problems in scientific computing are naturally posed on unbounded
domains. Typical examples arise in quantum mechanics, kinetic theory,
diffusion processes, fractional and nonlocal models, wave propagation, and
inverse problems. A common numerical treatment is to truncate the unbounded
domain and impose artificial, absorbing, or transparent boundary conditions on
the artificial boundary. This strategy makes it possible to use standard
bounded-domain discretizations, but the accuracy may depend sensitively on the
truncation size and the boundary treatment, especially for slowly decaying,
oscillatory, or nonlocal solutions \cite{Tsynkov1998,Antoine2008}.

Spectral methods based on basis functions defined on unbounded domains provide
an alternative direct approach. Hermite functions are natural on the whole real
line, while Laguerre functions are commonly used on the half-line. These
methods avoid artificial boundaries and can achieve high accuracy for smooth
functions with compatible decay. Classical and recent developments include
Laguerre and Hermite spectral methods for unbounded domains
\cite{Shen2000,ShenWang2009,ShenTangWang2011}, Hermite methods for parabolic
and fractional PDEs \cite{MaSunTang2005,MaoShen2017}, and generalized Hermite
approaches for fractional Laplacian and Schr\"odinger operators
\cite{ShengMaLiWangJia2021}. Generalized Hermite spectral methods have also
been applied to distributed-order time-fractional reaction--diffusion equations
on multi-dimensional unbounded domains \cite{GuoChenMeiSong2021}. Classical Hermite and Laguerre spectral methods usually rely on specially designed quadrature nodes and global basis systems, which may limit their applicability when only uniformly sampled or experimentally collected data are available.

A central issue in global Hermite and Laguerre approximations is the choice of
scaling, translation, and expansion order. Recent adaptive spectral methods
therefore emphasize parameter selection. Xia, Shao and Chou introduced scaling
and moving techniques based on frequency and exterior-error indicators
\cite{XiaShaoChou2021}. Chou, Shao and Xia further analyzed adaptive Hermite
spectral methods and related their effectiveness to frequency indicators and
$p$-refinement \cite{ChouShaoXia2023}. More recently, Hu and Yu developed a
bandwidth-type error analysis for scaled generalized Hermite and Laguerre
approximations, giving systematic guidance for scaling factors and convergence
behavior beyond the classical weighted Sobolev theory \cite{HuYu2026}. These
works show that the efficiency of unbounded-domain spectral approximations is
closely related to how well the basis scale and resolution match the spatial
and frequency content of the target function.

Frame-based extension approximation gives another useful viewpoint. Fourier
extension approximates a nonperiodic function on a finite interval by a
Fourier series on an enlarged interval. The resulting system is redundant and
ill-conditioned, but regularized Fourier extension approximations can be
stable and accurate in finite precision
\cite{AdcockHuybrechsMartinVaquero2014,AdcockRuan2014,MatthysenHuybrechs2016}.
Hermite extension follows the same extension philosophy with basis functions
defined on the whole real line. In \cite{Zhao2021HermiteExtension}, functions
given on arbitrary finite intervals were extended to $\mathbb R$ by a
regularized Hermite-function approximation and the resulting stable extension
was used for numerical differentiation. The common advantage of extension
frames is that redundancy provides flexibility: one may use a basis system
whose natural domain or analytic structure is more favorable than the original
computational interval.

Nevertheless, global frame approximations remain inefficient when the target
function contains localized oscillations, endpoint layers, or piecewise smooth
structures. Since all basis functions act on the whole interval, a local
irregularity or local high-frequency component may force a large global
expansion order. This motivates localized frame approximations. In our previous
work \cite{ZhaoWang2026LFE}, a local Fourier extension (LFE) method was
developed by partitioning the target interval into subintervals and applying
stabilized Fourier extension on each local piece. After mapping a subinterval
of length $h$ to a reference interval, an oscillation $e^{i\omega x}$ has an
effective frequency proportional to $\omega h$. Thus subdivision reduces the
local frequency scale and allows accurate reconstruction on small local
systems. The same localization also permits the local FE matrices and their
SVDs to be precomputed and reused. Related multi-interval FE constructions
have been used for numerical differentiation \cite{zhao2025fast}.

The localization mechanism is also useful for identifying internal
low-regularity points. In a related LFE quadrature study \cite{liu2026high},
windows containing derivative discontinuities were detected by an abnormal
increase of the local coefficient energy. One-sided LFE reconstructions were
then used to localize the singular point with subgrid accuracy. This feature is
important for the present work because the finite core may contain continuous
but piecewise smooth structures. Once an internal edge is detected and aligned
with the local partition, the LFE approximation is again applied on smooth
one-sided pieces.

The present paper extends the localized frame viewpoint to approximation on
unbounded domains. The key observation is that the finite core and the
far-field tails often have different numerical characters. The core may contain
nonperiodic oscillations, local layers, and low-regularity structures, for
which LFE is appropriate. The tails, by contrast, are often governed by a small
number of dominant oscillatory modes multiplied by slowly varying decaying
envelopes. It is therefore inefficient to use a single global Hermite or
Laguerre system for the whole real line. We instead use a core--tail frame
representation: LFE in the finite core and modulated Laguerre frames in the two
semi-infinite tails.

The domain is decomposed as
\[
    \mathbb R=(-\infty,a]\cup[a,b]\cup[b,\infty),
\]
where $[a,b]$ is the finite core. In the core, local Fourier extension frames
resolve nonperiodic, oscillatory, and locally nonsmooth structures. In each
tail, shifted and scaled Laguerre functions are multiplied by sine and cosine
modulation factors. The oscillatory carrier is represented by the modulation,
while the Laguerre functions approximate only the decaying envelope. The
modulation centers are estimated from local Fourier extension coefficients in
near-interface detection windows. Thus the LFE component has three roles: it
provides the core approximation, supplies tail frequency information, and
provides coefficient-energy indicators for internal edge detection.

The main contributions of this work are as follows. First, we develop a
localized core--tail Fourier--Laguerre frame approximation on the real line,
which combines local Fourier extension in the finite core with modulated
Laguerre frames in the decaying tails. Second, we exploit the local frequency
scaling property of LFE to construct an adaptive frequency-detection mechanism
for selecting effective tail modulation centers. Third, we incorporate an LFE
coefficient-energy based edge indicator, allowing continuous piecewise smooth
functions to be accurately approximated by aligning the local partition with
detected internal structures. Fourth, we establish an error framework that
separates the contributions from the core approximation error, Laguerre
envelope approximation error, residual phase caused by frequency mismatch,
finite-tail truncation, and TSVD regularization.

In addition, the proposed framework has two practical advantages over
traditional global spectral approximations. On the one hand, the approximation
is formulated through local least-squares problems and therefore does not
require problem-dependent Gaussian quadrature nodes, such as Gauss--Hermite or
Gauss--Laguerre points. The method can directly employ uniform sampling or
existing discrete data, which makes it more convenient for data-driven
scientific computing applications. On the other hand, the localization
strategy significantly reduces the computational cost. Instead of solving a
large global spectral system, the proposed method only requires a collection of
small local FE systems in the core and a few low-dimensional modulated
Laguerre systems in the tails. Consequently, the computational complexity grows
mainly with the number of local subintervals rather than with a global
expansion size.

Finally, we validate the proposed method through numerical experiments
including smooth and oscillatory approximation problems, derivative-jump
examples, comparisons with Hermite approximations, and a decaying model problem
on the real line.

The rest of this paper is organized as follows. Section~\ref{sec:lfe}
reviews the LFE approximation, internal edge detection, and adaptive frequency
detection. Section~\ref{sec:tail} introduces the modulated Laguerre frame on
semi-infinite tails. Section~\ref{sec:approximation} presents approximation
and error estimates. Section~\ref{sec:param} discusses parameter selection and
computational complexity. Numerical experiments are reported in
Section~\ref{sec:numerical}. Section~\ref{sec:conclusion} concludes the paper.

\section{Local Fourier extension approximation and adaptive frequency detection}
\label{sec:lfe}

In this section, we describe the local Fourier extension (LFE)
approximation used in the finite core region and explain how it provides
structural and frequency information for the proposed core--tail
approximation. The role of LFE in the present method is threefold. First,
it gives a localized frame approximation for nonperiodic and oscillatory
functions on bounded intervals. Second, the local coefficient information
can be used to detect internal low-regularity points, such as derivative
discontinuities, inside the core region. Third, the adaptive local scale
near the core--tail interfaces is used to determine the frequency range in
which the tail modulation frequencies are detected.

\subsection{Fourier extension on a finite interval}

Let \(I=[x_0,x_1]\) be a finite interval and let \(f\) be sampled at
equispaced points in \(I\). We map \(I\) to the reference interval
\([-1,1]\) by
\[
    t=\frac{2(x-x_c)}{h},
    \qquad
    x_c=\frac{x_0+x_1}{2},
    \qquad
    h=x_1-x_0 .
\]
For a fixed extension parameter \(T>1\), the Fourier extension
approximation of order \(N\) is written as
\[
    f_N(x)
    =
    \sum_{\ell=-N}^{N}c_\ell
    \exp\left(\frac{i\pi \ell t(x)}{T}\right).
    \label{eq:fe-expansion}
\]
The mode \(\ell\) corresponds to the physical frequency
\[
    \kappa_\ell
    =
    \frac{2\pi\ell}{Th}.
    \label{eq:physical-frequency}
\]
Thus, for a fixed local order \(N\), decreasing the interval length \(h\)
increases the physical frequency range that can be represented.

Let
\[
    x_j=x_0+\frac{j-1}{m-1}(x_1-x_0),
    \qquad j=1,\ldots,m,
\]
be the sampling points and let \(f_j=f(x_j)\). The discrete Fourier
extension system is
\[
    A c \approx f,
\]
where
\[
    A_{j,\ell}
    =
    \exp\left(\frac{i\pi \ell t(x_j)}{T}\right),
    \qquad
    \ell=-N,\ldots,N.
\]
The Fourier extension system is redundant and may be ill-conditioned.
Therefore, the coefficient vector is computed by truncated singular value
decomposition (TSVD). If
\[
    A=U\Sigma V^*
\]
is the singular value decomposition, then for a threshold \(\epsilon>0\)
we use
\[
    c^\epsilon
    =
    \sum_{\sigma_j>\epsilon}
    \frac{u_j^*f}{\sigma_j}v_j .
    \label{eq:tsvd-fe}
\]
The regularized approximation is then
\[
    f_N^\epsilon(x)
    =
    \sum_{\ell=-N}^{N}c_\ell^\epsilon
    \exp\left(\frac{i\pi \ell t(x)}{T}\right).
\]
The redundancy of the frame is useful because it allows nonperiodic
functions on \(I\) to be approximated by a Fourier representation on a
larger interval, while TSVD suppresses the unstable components associated
with small singular values.

\subsection{Local Fourier extension in the core region}

The finite core interval \([a,b]\) is partitioned into subintervals
\[
    a=x_0<x_1<\cdots<x_K=b,
    \qquad
    I_k=[x_{k-1},x_k],
    \quad k=1,\ldots,K .
\]
On each \(I_k\), we construct an independent Fourier extension
approximation
\[
    f_k^\epsilon(x)
    =
    \sum_{\ell=-N}^{N}c_{k,\ell}^{\epsilon}
    \exp\left(\frac{i\pi \ell t_k(x)}{T}\right),
    \qquad x\in I_k,
    \label{eq:local-fe}
\]
where
\[
    t_k(x)
    =
    \frac{2(x-x_{c,k})}{h_k},
    \qquad
    x_{c,k}=\frac{x_{k-1}+x_k}{2},
    \qquad
    h_k=x_k-x_{k-1}.
\]
The core approximation is defined piecewise by
\[
    f_C^\epsilon(x)
    =
    f_k^\epsilon(x),
    \qquad x\in I_k .
\]

The localization is essential for oscillatory functions. If
\(f(x)=e^{i\omega x}\), then on \(I_k\) the mapped function is
\[
    e^{i\omega x}
    =
    e^{i\omega x_{c,k}}
    e^{i(\omega h_k/2)t_k}.
\]
Hence the effective frequency in the reference variable is proportional to
\(\omega h_k\). By reducing \(h_k\), a highly oscillatory function in the
physical variable becomes a moderate-frequency function in the local
reference variable. This is the basic reason why LFE can resolve
high-frequency structures with a small local order \(N\).

The largest physical frequency represented by the local FE basis is
approximately
\[
    \kappa_{\max,k}
    \approx
    \frac{2\pi N}{T h_k}.
    \label{eq:local-kmax}
\]
Therefore, a practical local resolution condition is
\[
    \frac{2\pi N}{T h_k}
    \gtrsim
    \omega_{\max,k},
    \label{eq:lfe-resolution}
\]
where \(\omega_{\max,k}\) denotes the largest significant frequency in
\(I_k\). For a uniform partition with \(h=(b-a)/K\), this gives the
guideline
\[
    K
    \gtrsim
    \frac{(b-a)T\omega_{\max}}{2\pi N}.
\]
In the actual implementation, the partition may also be generated
adaptively by monitoring the local residual. When the residual on a
subinterval is not sufficiently small, the subinterval is divided. This
adaptive process produces smaller intervals in regions with stronger
oscillation or lower regularity.

Another computational advantage of the local construction is that all
subintervals can be mapped to the same reference interval. Thus the local
Fourier extension matrices and their singular value decompositions can be
precomputed and reused whenever the same values of \(m\), \(N\), and \(T\)
are used. This makes the local frame approximation efficient even when the
number of subintervals is moderately large.

\subsection{Internal edge detection in the core region}
\label{subsec:internal-edge-detection}

The local Fourier extension coefficients also provide useful structural
information inside the core region. In particular, when a function is
continuous but has a derivative discontinuity or a localized low-regularity
point, the local window containing this point is no longer fitted from a
single smooth branch. This usually leads to an abnormal growth of the TSVD
coefficient vector. Therefore, for the \(k\)-th local window, we introduce
the coefficient-energy indicator
\[
    \eta_k=\|c_k^\epsilon\|_2 .
    \label{eq:coefficient-energy-indicator}
\]
On smooth windows, \(\eta_k\) remains moderate, whereas a window containing
an internal edge typically produces a pronounced outlier.

A detailed edge-detection and localization strategy based on this idea has
been developed in our related LFE quadrature work~\cite{liu2026high}. There, singularity-containing windows are first detected by the coefficient-energy indicator;
the grid cell containing the singular point is then localized by comparing
one-sided LFE coefficient energies; finally, one-sided local Fourier
extension models are used to estimate the singular point with subgrid
accuracy. The method is designed for continuous piecewise smooth functions,
especially those with derivative discontinuities, and can recover high
accuracy after the local partition is aligned with the detected point.

In the present work, we only use this mechanism as a lightweight structural
tool in the core region. If an internal edge is detected, the core partition
is refined or aligned with the estimated edge location, so that subsequent
local Fourier extension approximations are constructed on one-sided smooth
pieces. This allows the core--tail framework to handle functions that are
continuous but not globally smooth in the finite core.

\subsection{Adaptive frequency detection near the interfaces}

The same local scale information is used to determine the frequency range
for the tail modulation. We describe the right interface \(x=b\); the left
interface \(x=a\) is treated analogously.

Let
\[
    I_{\rm end}=[b-h_{\rm end},b]
\]
denote the terminal subinterval adjacent to the right interface $x=b$,
where $h_{\rm end}$ is the length of the last local LFE element generated by
the adaptive partition. The local FE approximation on this interval uses
$N$ Fourier modes and an extension parameter $T$. After mapping the physical
interval $I_{\rm end}$ onto the reference interval $[-1,1]$, the highest
physical frequency that can be represented by the local FE basis is
approximately
\[
    \omega_{\rm loc}
    \approx
    \frac{2\pi N}{T h_{\rm end}},
    \label{eq:terminal-frequency-scale}
\]
where $\omega_{\rm loc}$ denotes the effective local frequency resolution of
the terminal LFE block. Therefore, $h_{\rm end}$ provides a practical
indicator of the local oscillatory scale near the interface: a smaller
$h_{\rm end}$ usually corresponds to a more oscillatory or less regular
interface neighborhood, while a larger $h_{\rm end}$ indicates a smoother
local structure.

To detect the dominant oscillatory components near the interface, we introduce
a local frequency-detection window
\[
    W_R=[b-L_{\rm det},b],
\]
where $L_{\rm det}$ is the detection-window length. Instead of using a fixed
global window, we determine it adaptively according to the terminal LFE scale:
\[
    L_{\rm det}
    =
    \min\left\{
        L_{\max},
        \max\left\{
            L_{\min},
            r_{\rm det}h_{\rm end}
        \right\}
    \right\}.
    \label{eq:Ldet-choice}
\]
Here, $L_{\min}$ and $L_{\max}$ are prescribed lower and upper bounds for the
detection window size, respectively, and $r_{\rm det}$ is a dimensionless
window enlargement factor controlling how many neighboring LFE elements are
included in the frequency analysis.

The Fourier detection system on $W_R$ uses an extension parameter
$T_{\rm det}$ and $N_{\rm det}$ Fourier modes. The corresponding maximum
detectable physical frequency is
\[
    \omega_{\max}^{\rm det}
    =
    \frac{2\pi N_{\rm det}}
         {T_{\rm det}L_{\rm det}} .
\]
To ensure that the detection system can resolve the oscillatory scale already
resolved by the terminal LFE element, we require
\[
    \frac{2\pi N_{\rm det}}
         {T_{\rm det}L_{\rm det}}
    \ge
    s_\omega
    \frac{2\pi N}
         {T h_{\rm end}},
    \label{eq:Ndet-condition}
\]
where $s_\omega>1$ is a safety factor introduced to provide additional
frequency coverage.

Consequently, the detection order is selected as
\[
    N_{\rm det}
    =
    \left\lceil
    \frac{T_{\rm det}L_{\rm det}}{2\pi}
    s_\omega
    \frac{2\pi N}{T h_{\rm end}}
    \right\rceil ,
    \label{eq:Ndet-choice}
\]
subject to prescribed minimum and maximum bounds:
\[
    N_{\min}^{\rm det}
    \le N_{\rm det}
    \le N_{\max}^{\rm det}.
\]
In the common case where the detection window has the same scale as the
terminal LFE element, namely
\[
    T_{\rm det}=T,\qquad
    L_{\rm det}=O(h_{\rm end}),
\]
the detection order satisfies
\[
    N_{\rm det}=O(N).
\]
Therefore, the frequency-detection procedure introduces only a local
computational cost comparable with one LFE block, rather than a cost determined
by the global carrier frequency of the original function.

Applying Fourier extension on \(W_R\) gives coefficients
\[
    c_{\ell}^{R,\epsilon},
    \qquad
    \ell=-N_{\rm det},\ldots,N_{\rm det}.
\]
The physical frequency associated with the positive mode \(\ell>0\) is
\[
    \kappa_\ell^R
    =
    \frac{2\pi\ell}{T_{\rm det}L_{\rm det}}.
\]
We define the frequency indicator
\[
    E_R(\kappa_\ell^R)
    =
    |c_{\ell}^{R,\epsilon}|
    +
    |c_{-\ell}^{R,\epsilon}|,
    \qquad
    \ell=1,\ldots,N_{\rm det}.
    \label{eq:freq-indicator}
\]
Dominant local peaks of \(E_R\) are used as candidate modulation centers
for the right tail. In the present implementation, a peak is retained if
it is above a relative threshold and sufficiently separated from previously
selected peaks. The resulting frequency set is denoted by
\[
    \mathcal K_R=\{\kappa_1^R,\ldots,\kappa_{q_R}^R\}.
\]
The left-tail frequency set
\[
    \mathcal K_L=\{\kappa_1^L,\ldots,\kappa_{q_L}^L\}
\]
is constructed in the same way from a detection window near \(a\).

The frequencies in \(\mathcal K_L\) and \(\mathcal K_R\) are not required
to be exact carrier frequencies. They serve as modulation centers for
dominant frequency bands. Any moderate mismatch between a true frequency
and its detected center becomes a residual oscillation in the tail
envelope, which can be absorbed by the Laguerre approximation. This point
is quantified in Section~\ref{sec:approximation}.

\section{Modulated Laguerre frame approximation on semi-infinite tails}
\label{sec:tail}

In this section, we introduce the modulated Laguerre frame approximation
used in the two semi-infinite tails. The purpose is to exploit the natural
half-line structure of Laguerre functions while avoiding the need for
Laguerre functions to resolve the full oscillatory tail. The oscillatory
carriers are represented by sine and cosine modulation factors, whereas
the Laguerre functions approximate the slowly varying decaying envelopes.

\subsection{Laguerre functions and tail mappings}

Let \(L_m(s)\) be the Laguerre polynomial of degree \(m\). We use the
Laguerre functions
\[
    \ell_m(s)=e^{-s/2}L_m(s),
    \qquad s\ge0,
    \qquad m=0,1,\ldots .
\]
They are evaluated by the standard recurrence
\[
    L_0(s)=1,\qquad L_1(s)=1-s,
\]
and
\[
    (m+1)L_{m+1}(s)
    =
    (2m+1-s)L_m(s)-mL_{m-1}(s),
    \qquad m\ge1.
\]

On the right tail \([b,\infty)\), we introduce
\[
    s_R=\alpha_R(x-b),
    \qquad x\ge b,
    \label{eq:right-tail-map}
\]
where \(\alpha_R>0\) is a scaling factor. On the left tail
\((-\infty,a]\), we use
\[
    s_L=\alpha_L(a-x),
    \qquad x\le a.
    \label{eq:left-tail-map}
\]
Thus both tails are mapped to the half-line \(s\ge0\). The scaling
parameters \(\alpha_R\) and \(\alpha_L\) adapt the Laguerre variable to the
physical decay scale of the tail envelopes.

\subsection{Single-frequency modulation}

Consider a right-tail component of the form
\[
    f_R(x)\approx A_R(x)\sin(\omega x+\varphi),
    \qquad x\ge b,
\]
where \(A_R(x)\) is a slowly varying decaying envelope. A standard
Laguerre expansion must approximate both \(A_R(x)\) and the oscillatory
carrier. Instead, we use the modulated Laguerre frame
\[
    \ell_m(s_R)\sin(\kappa x),
    \qquad
    \ell_m(s_R)\cos(\kappa x),
    \qquad m=0,\ldots,M_R,
\]
where \(\kappa\) is a modulation center. The corresponding approximation
is
\[
    f_R(x)
    \approx
    \sum_{m=0}^{M_R}a_m^R\ell_m(s_R)\sin(\kappa x)
    +
    \sum_{m=0}^{M_R}b_m^R\ell_m(s_R)\cos(\kappa x).
    \label{eq:single-tail-real}
\]
Equivalently,
\[
    f_R(x)
    \approx
    \operatorname{Re}\left\{
    e^{i\kappa x}
    \sum_{m=0}^{M_R}d_m^R\ell_m(s_R)
    \right\}.
    \label{eq:single-tail-complex}
\]

If the true carrier frequency is \(\omega\) and
\(\kappa=\omega+\delta\), then
\[
    A_R(x)e^{i\omega x}
    =
    e^{i\kappa x}
    \left(A_R(x)e^{-i\delta x}\right).
    \label{eq:freq-mismatch-envelope1}
\]
Therefore, a frequency mismatch does not directly destroy the
representation. It only modifies the envelope by the residual factor
\(e^{-i\delta x}\). In the scaled variable \(s_R=\alpha_R(x-b)\), this
residual oscillation has frequency \(|\delta|/\alpha_R\). Hence the effect
of a moderate frequency mismatch can be compensated by the Laguerre
envelope approximation.

The same construction is used on the left tail with \(s_L=\alpha_L(a-x)\).

\subsection{Multi-frequency modulated Laguerre frames}

In many problems, the far field is not strictly single-frequency, but it
is often frequency-sparse. Let
\[
    \mathcal K_R=\{\kappa_1^R,\ldots,\kappa_{q_R}^R\}
\]
be the modulation centers selected for the right tail. We approximate the
right tail by
\[
    f_R(x)
    \approx
    \sum_{p=1}^{q_R}
    \left[
    \sum_{m=0}^{M_R}a_{p,m}^R
    \ell_m(s_R)\sin(\kappa_p^R x)
    +
    \sum_{m=0}^{M_R}b_{p,m}^R
    \ell_m(s_R)\cos(\kappa_p^R x)
    \right].
    \label{eq:right-tail-mf}
\]
Similarly, for the left-tail frequency set
\[
    \mathcal K_L=\{\kappa_1^L,\ldots,\kappa_{q_L}^L\},
\]
we use
\[
    f_L(x)
    \approx
    \sum_{p=1}^{q_L}
    \left[
    \sum_{m=0}^{M_L}a_{p,m}^L
    \ell_m(s_L)\sin(\kappa_p^L x)
    +
    \sum_{m=0}^{M_L}b_{p,m}^L
    \ell_m(s_L)\cos(\kappa_p^L x)
    \right].
    \label{eq:left-tail-mf}
\]

All selected frequencies in one tail are fitted jointly. This is more
stable than fitting each frequency band separately, since the
least-squares problem automatically distributes the data among the
modulated components. The number of unknowns in one tail is
\[
    n_{\rm tail}=2q(M+1),
    \label{eq:tail-dof}
\]
where \(q\) is the number of selected modulation centers and \(M\) is the
Laguerre order. Since the far field is assumed to be frequency-sparse,
\(q\) is small and the tail system remains of moderate size.

\subsection{Discrete least-squares formulation}

We describe the right-tail system. Let
\[
    b=x_1^R<x_2^R<\cdots<x_{m_R}^R=b+L_R
\]
be sample points on the finite tail window \([b,b+L_R]\). The window
length \(L_R\) is chosen so that the remaining far-field magnitude beyond
\(b+L_R\) is negligible or irrelevant for the computational task. Set
\[
    s_j^R=\alpha_R(x_j^R-b),
    \qquad j=1,\ldots,m_R .
\]
For each frequency \(\kappa_p^R\), define
\[
    S_{j,m}^{(p)}
    =
    \ell_m(s_j^R)\sin(\kappa_p^R x_j^R),
    \qquad
    C_{j,m}^{(p)}
    =
    \ell_m(s_j^R)\cos(\kappa_p^R x_j^R),
\]
for \(m=0,\ldots,M_R\). The full right-tail matrix is
\[
    A_R=
    \left[
    S^{(1)},C^{(1)},S^{(2)},C^{(2)},\ldots,
    S^{(q_R)},C^{(q_R)}
    \right].
\]
Given the data vector
\[
    f_R=
    \left(f(x_1^R),\ldots,f(x_{m_R}^R)\right)^T,
\]
the tail coefficients are computed from
\[
    A_Rc_R\approx f_R .
\]
As for the LFE approximation, we use TSVD regularization:
\[
    c_R^\epsilon
    =
    \sum_{\sigma_j>\epsilon}
    \frac{u_j^*f_R}{\sigma_j}v_j,
    \qquad
    A_R=U\Sigma V^* .
    \label{eq:tail-tsvd}
\]
The left-tail system is constructed analogously on \([a-L_L,a]\).

The parameters \(M\), \(\alpha\), \(L\), and \(q\) determine the accuracy
and cost of the tail approximation. Their influence will be examined
separately in the parameter study and complexity analysis in the next
section.

\section{Approximation properties of the core--tail representation}
\label{sec:approximation}

This section gives error estimates for the proposed core--tail
representation. The purpose is to make explicit how the approximation error is
split among the core LFE error, the modulated Laguerre envelope error, the
residual phase caused by frequency mismatch, the finite-window truncation
error, and the TSVD regularization error.

Let
\[
    \mathbb R=(-\infty,a]\cup[a,b]\cup[b,\infty)
\]
and define the piecewise approximation
\[
 f_{
m app}(x)=
 \begin{cases}
 f_L^\epsilon(x), & x\le a,\\
 f_C^\epsilon(x), & a\le x\le b,\\
 f_R^\epsilon(x), & x\ge b.
 \end{cases}
\]
For any norm that is additive or subadditive over the three parts, such as a
continuous or discrete $L^2$ norm, the total error satisfies
\begin{equation}
\|f-f_{\rm app}\|
\le
\|f-f_L^\epsilon\|_{(-\infty,a]}
+
\|f-f_C^\epsilon\|_{[a,b]}
+
\|f-f_R^\epsilon\|_{[b,\infty)} .
\label{eq:basic-error-splitting}
\end{equation}
We next estimate the three terms separately.

\subsection{Core approximation on smooth pieces}

Assume that the core partition is
\[
    a=x_0<x_1<\cdots<x_K=b,
    \qquad I_k=[x_{k-1},x_k],
\]
and that each element is mapped to the reference interval $[-1,1]$. We first
consider the case where $f$ is smooth on every $I_k$. Let $h_k=|I_k|$ and let
$f_k^\epsilon$ be the TSVD-stabilized local Fourier extension approximation of
order $N$.

\begin{pro}[Local core error]
\label{pro:local-core-error}
Let $f\in H^s(I_k)$, $s>1/2$, on a subinterval $I_k$. For fixed extension
parameter $T>1$ and TSVD threshold $\epsilon$, the local FE approximation
satisfies
\begin{equation}
    \|f-f_k^\epsilon\|_{L^2(I_k)}
    \le
    C_T h_k^{1/2}
    \inf_{p\in\Phi_N}
    \|\tilde f_k-p\|_{L^2(-1,1)}
    +C_T\epsilon\|c_k\|_2,
\label{eq:local-fe-error-general}
\end{equation}
where $\tilde f_k$ is the mapped function on $[-1,1]$, $\Phi_N$ is the
Fourier extension space, and $c_k$ denotes a coefficient vector of a stable
near-best approximation. In particular, if $\tilde f_k\in H^s(-1,1)$, then
\begin{equation}
    \|f-f_k^\epsilon\|_{L^2(I_k)}
    \le
    C_T h_k^{1/2} N^{-s}\|\tilde f_k\|_{H^s(-1,1)}
    +C_T\epsilon\|c_k\|_2 .
\label{eq:local-fe-algebraic}
\end{equation}
\end{pro}

\begin{proof}
The TSVD solution is quasi-optimal up to the discarded singular directions, so
its error is bounded by the best approximation error in the retained FE space
plus the TSVD perturbation term. The factor $h_k^{1/2}$ follows from the affine
mapping between $I_k$ and the reference interval. The algebraic bound follows
from the standard trigonometric approximation estimate on the reference
interval.
\end{proof}

The estimate shows explicitly how localization improves the resolution of
oscillatory functions. If $f(x)=e^{i\omega x}$ on $I_k$, the mapped function
contains the nondimensional frequency $\omega h_k/2$. Thus, for fixed $N$,
accuracy is controlled by the local frequency-length product rather than by
the global frequency. A practical resolution condition is
\begin{equation}
    \frac{2\pi N}{T h_k}
    \gtrsim \omega_{\max,k} .
\label{eq:core-resolution-condition}
\end{equation}

For continuous piecewise smooth functions, the same estimate applies on each
smooth piece. If a derivative discontinuity lies inside an element, the local
regularity assumption fails and spectral-type convergence is lost on that
element. This is the reason for using the internal edge detector described in
Section~\ref{subsec:internal-edge-detection}. Once the detected edge is aligned
with the partition, the local estimates again apply on the one-sided smooth
pieces.

\begin{corollary}[Core error after edge alignment]
\label{cor:edge-aligned-core}
Suppose that $f$ is continuous on $[a,b]$ and is piecewise $H^s$ with a finite
number of derivative discontinuities. If the core partition contains all true
or accurately detected breakpoints, then
\begin{equation}
    \|f-f_C^\epsilon\|_{L^2(a,b)}
    \le
    C_T
    \left(
    \sum_{k=1}^{K} h_k N^{-2s}
    \|\tilde f_k\|_{H^s(-1,1)}^2
    \right)^{1/2}
    +C_T\epsilon
    \left(
    \sum_{k=1}^{K}\|c_k\|_2^2
    \right)^{1/2} .
\label{eq:edge-aligned-core-error}
\end{equation}
If a detected breakpoint has location error $\rho$, then an additional local
term of order $O(\rho)$ is introduced for Lipschitz continuous functions with a
jump in the first derivative.
\end{corollary}

\subsection{Tail approximation with exact modulation centers}

We describe the right tail; the left tail is analogous. Suppose that on
$[b,\infty)$ the tail can be written as
\begin{equation}
    f_R(x)=
    \sum_{p=1}^{q_R}
    \operatorname{Re}\left\{e^{i\kappa_p^R x}A_p^R(x)\right\},
    \qquad x\ge b,
\label{eq:exact-modulated-tail}
\end{equation}
where the envelopes $A_p^R$ are slowly varying and decaying. With
$s=\alpha_R(x-b)$, define
\[
    \widehat A_p^R(s)=A_p^R(b+s/\alpha_R).
\]
Let $\Pi_M^{\rm Lag}$ denote the approximation operator onto
$\operatorname{span}\{\ell_0,\ldots,\ell_M\}$ in the scaled Laguerre variable.

\begin{pro}[Tail envelope error]
\label{pro:tail-envelope-error}
Assume that the modulation centers in \eqref{eq:exact-modulated-tail} are
known exactly. Then the modulated Laguerre approximation on the right tail
satisfies
\begin{equation}
    \|f_R-f_{R,M}\|_{[b,\infty)}
    \le
    C
    \sum_{p=1}^{q_R}
    \left\|
    \widehat A_p^R-
    \Pi_M^{\rm Lag}\widehat A_p^R
    \right\|_{[0,\infty)} .
\label{eq:tail-envelope-error}
\end{equation}
Consequently, the convergence is governed by the Laguerre approximability of
the envelopes rather than by the carrier frequencies.
\end{pro}

This estimate explains the advantage of modulation. Without modulation, a
Laguerre expansion must approximate both the envelope and the high-frequency
carrier. With modulation, the Laguerre basis only resolves the residual
envelope.

In computation the semi-infinite tail is replaced by a finite interval
$[b,b+L_R]$. Therefore,
\begin{equation}
    \|f_R-f_{R,M}^{\rm app}\|_{[b,\infty)}
    \le
    \|f_R-f_{R,M}^{\rm app}\|_{[b,b+L_R]}
    +\|f_R\|_{[b+L_R,\infty)} .
\label{eq:tail-window-splitting}
\end{equation}
For exponentially decaying tails, the second term decreases exponentially in
$L_R$.

\subsection{Frequency mismatch and residual phase}

The detected modulation centers do not have to be exact. Suppose that a true
tail component is $A(x)e^{i\omega x}$, but the selected modulation center is
$\kappa=\omega+\delta$. Then
\begin{equation}
    A(x)e^{i\omega x}
    =e^{i\kappa x}\left(A(x)e^{-i\delta x}\right).
\label{eq:freq-mismatch-envelope}
\end{equation}
Thus the frequency mismatch changes only the envelope. In the scaled variable
$s=\alpha(x-b)$,
\[
    e^{-i\delta x}=e^{-i\delta b}e^{-i(\delta/\alpha)s}.
\]
The relevant nondimensional mismatch parameter is therefore
\begin{equation}
    \nu=\frac{|\delta|}{\alpha}.
\label{eq:mismatch-parameter}
\end{equation}
Equivalently, on a finite tail window of length $L$, the accumulated residual
phase is $|\delta|L$.

\begin{pro}[Tail error with detected modulation centers]
\label{pro:mismatched-tail-error}
Let the true tail frequencies $\omega_j$ be assigned to detected modulation
centers $\kappa_{\pi(j)}$, and set $\delta_j=\omega_j-\kappa_{\pi(j)}$. Then
the modulated Laguerre approximation satisfies the abstract bound
\begin{equation}
    \|f_R-f_{R,M}^{\delta}\|_{[b,\infty)}
    \le
    C\sum_j
    \left\|
    \widehat A_j(s)e^{i(\delta_j/\alpha)s}
    -
    \Pi_M^{\rm Lag}
    \left(\widehat A_j(s)e^{i(\delta_j/\alpha)s}\right)
    \right\|_{[0,\infty)} .
\label{eq:mismatched-tail-error}
\end{equation}
Thus the influence of frequency detection enters through the residual
parameters $|\delta_j|/\alpha$, not through the original frequencies
$|\omega_j|$.
\end{pro}

A useful consequence concerns close frequencies. If two components have
frequencies $\omega_1$ and $\omega_2$, then
\[
    e^{i\omega_1x}+e^{i\omega_2x}
    =2e^{i\bar\omega x}
    \cos\left(\frac{\Delta x}{2}\right),
    \qquad
    \bar\omega=\frac{\omega_1+\omega_2}{2},
    \quad \Delta=\omega_2-\omega_1 .
\]
If both components are represented by the single modulation center
$\bar\omega$, the beat frequency becomes part of the envelope, with
nondimensional scale
\begin{equation}
    \nu_{\rm close}=\frac{|\Delta|}{2\alpha}.
\end{equation}
Therefore close frequencies need not always be separated by the detector,
provided that the residual beat can be resolved by the Laguerre envelope.

\subsection{A combined error estimate}

Combining the preceding estimates gives the following summary bound.

\begin{theorem}[Core--tail approximation error]
\label{thm:core-tail-error}
Assume that the core partition is aligned with all internal breakpoints, or
with detected breakpoints whose maximum location error is $\rho$. Assume also
that each tail has a frequency-sparse representation with decaying envelopes,
and let the detected modulation centers have residual parameters
$\nu_j=|\omega_j-\kappa_{\pi(j)}|/\alpha$. Then, up to constants depending on
$T$ and on the frame bounds of the discrete least-squares systems,
\begin{equation}
\begin{aligned}
\|f-f_{\rm app}\|
\lesssim
&\; E_{\rm core}(N,K,T;f)
+E_{\rm Lag}^L(M_L,\alpha_L,\{\nu_j^L\})
+E_{\rm Lag}^R(M_R,\alpha_R,\{\nu_j^R\})  \\
&\; +E_{\rm win}^L+E_{\rm win}^R
+E_{\rm TSVD}
+E_{\rm edge}(\rho).
\end{aligned}
\label{eq:total-error-estimate}
\end{equation}
Here $E_{\rm core}$ is the sum of local FE errors on the smooth core pieces,
$E_{\rm Lag}^{L,R}$ are the Laguerre approximation errors of the residual
envelopes after modulation, $E_{\rm win}^{L,R}$ are the finite-window tail
truncation errors, $E_{\rm TSVD}$ is the regularization error, and
$E_{\rm edge}(\rho)=O(\rho)$ for Lipschitz continuous functions with derivative
jumps.
\end{theorem}

\begin{proof}
The estimate follows from the basic splitting
\eqref{eq:basic-error-splitting}. The core term is bounded by
Proposition~\ref{pro:local-core-error} and Corollary~\ref{cor:edge-aligned-core}.
The two tail terms are bounded by Proposition~\ref{pro:tail-envelope-error},
with the modification described in Proposition~\ref{pro:mismatched-tail-error}
when detected frequencies are not exact. The finite-tail window and TSVD terms
are then added by the triangle inequality. If the detected edge is displaced by
$\rho$, only an interval of length $O(\rho)$ is assigned to the wrong smooth
branch, giving the stated $O(\rho)$ contribution for Lipschitz data.
\end{proof}

The theorem clarifies the role of the proposed construction. The core error is
controlled by local smoothness and local frequency-length products; the tail
error is controlled by the approximability of the residual envelopes; and
frequency detection only needs to make the residual phase moderate. These
observations motivate the parameter choices and numerical tests in the next two
sections.

\section{Parameter selection and complexity analysis}
\label{sec:param}

\subsection{Scaling rule for the Laguerre tail approximation}

The scaling parameter determines how the physical tail interval is mapped
to the Laguerre coordinate. For a tail interval of length \(L\), we introduce
\[
    s=\alpha x,\qquad x\in[0,L].
\]
Thus the physical interval \([0,L]\) is mapped to \([0,\alpha L]\) in the
Laguerre variable. If \(\alpha\) is too small, the Laguerre coordinate range
is too short to resolve the variation of the envelope on the whole tail
window. If \(\alpha\) is too large, the effective Laguerre coordinate range
becomes unnecessarily long and the least-squares approximation may become
less stable. Hence \(\alpha\) should be coupled with both the tail length
\(L\) and the Laguerre order \(M\).

We write the scaling parameter in the form
\[
    \alpha=\frac{CM}{L},
\]
where \(C\) is a dimensionless constant. Equivalently,
\[
    \alpha L=CM .
\]
This means that, as the Laguerre order increases, the effective
computational range in the Laguerre coordinate also increases
proportionally. The parameter \(C\) controls how rapidly this effective
range grows with \(M\).

To select a robust value of \(C\), we consider the envelope approximation
problem
\[
    g(x)=e^{-\beta x},\qquad x\in[0,L],
\]
for several decay rates \(\beta\). The approximation is represented as
\[
    g_M(x)=\sum_{m=0}^{M}c_m e^{-\alpha x/2}L_m(\alpha x),
\]
where the coefficients are computed by a TSVD-regularized least-squares
procedure.

Figure~\ref{fig:scaling_rule} shows the worst-case approximation error over
the tested decay rates in the \((C,M)\) parameter plane. The results indicate
that very small values of \(C\) are not optimal for moderate Laguerre orders,
because the mapped interval \([0,\alpha L]\) is too short. On the other hand,
excessively large values of \(C\) may also deteriorate the approximation. A
stable low-error region is observed for moderate values of \(C\), especially
around
\[
    C\approx 0.8 .
\]
This region is particularly relevant for the present algorithm, because the
tail least-squares cost depends strongly on the Laguerre order. Indeed, if
\(q\) modulation frequencies are used, the number of tail unknowns is
approximately
\[
    n_t=2q(M+1),
\]
and a direct SVD solve scales cubically with \(n_t\). Therefore, it is
preferable to use a moderate Laguerre order together with a scaling parameter
that gives sufficient resolution in the tail coordinate.

Accordingly, in the following experiments we adopt
\[
    C=\frac45,
    \qquad
    \boxed{\alpha=\frac{4M}{5L}}.
\]
With this choice, \(M=25\)--\(30\) already gives a sufficiently large
effective Laguerre coordinate range for the tested tail envelopes. Compared
with using a larger Laguerre order with a more conservative scaling, this
choice reduces the dimension of the tail least-squares systems while
maintaining high accuracy.

\begin{figure}[htbp]
\centering
\includegraphics[width=0.65\textwidth]{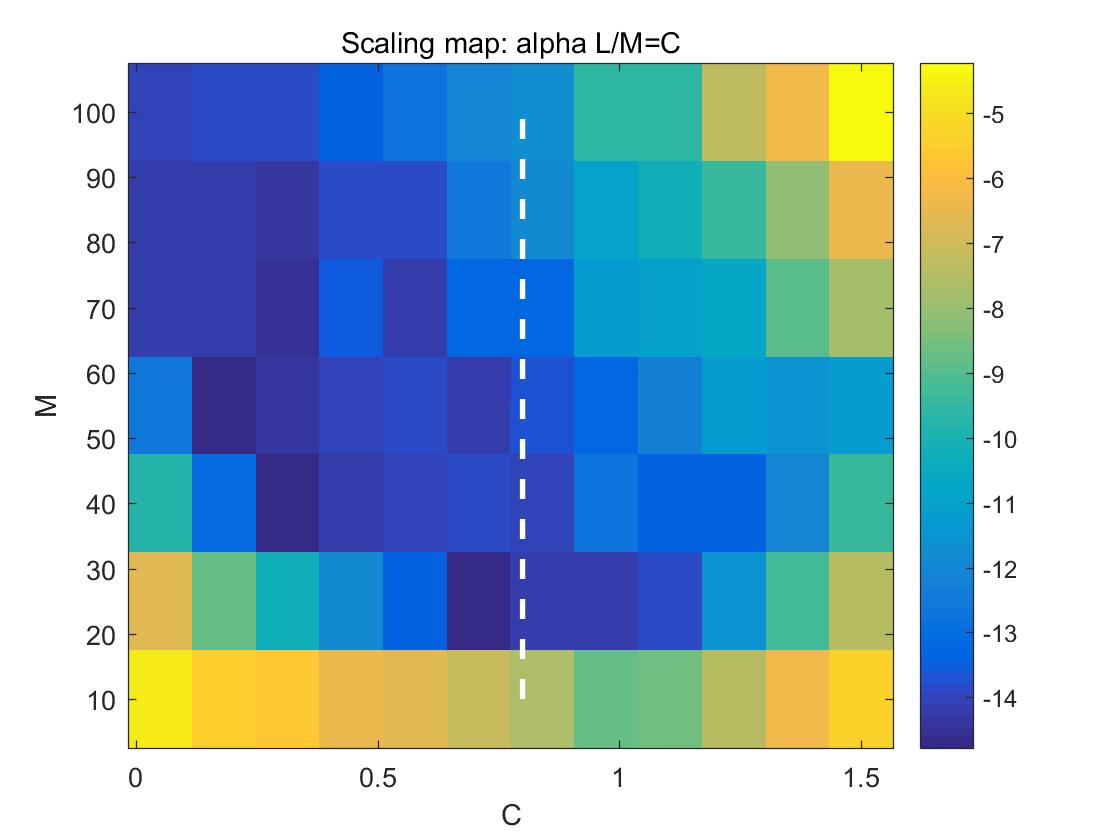}
\caption{
Scaling map for the rule \(\alpha=CM/L\). The color represents
\(\log_{10}(\max_\beta E)\), where \(E\) is the maximum approximation error
over different decay rates. The dashed line indicates \(C=4/5\), which lies
in a stable low-error region for moderate Laguerre orders and is adopted in
the subsequent computations.
}
\label{fig:scaling_rule}
\end{figure}

\subsection{Influence of the residual modulation frequency}

After the scaling rule has been fixed as
\[
    \alpha=\frac{4M}{5L},
\]
we examine how the accuracy of the modulation frequency affects the
required Laguerre order. Suppose that a tail component has carrier frequency
\(\omega\), while the detected modulation center is \(\kappa\). After
modulation, the remaining oscillation in the envelope is governed by the
residual frequency
\[
    \delta=|\omega-\kappa|.
\]
On a tail interval of length \(L\), the accumulated residual phase is
\[
    \Lambda=\delta L=|\omega-\kappa|L .
\]
This quantity measures how many residual oscillations must still be
resolved by the Laguerre envelope after the dominant carrier has been
removed. It plays a role analogous to the local frequency-length product in
local Fourier extension methods.

To isolate this effect, we consider the model envelope problem
\[
    g(x)=e^{-\beta x}\cos(\delta x),\qquad x\in[0,L],
\]
with several decay rates \(\beta\). For each value of
\(\Lambda=\delta L\), we increase the Laguerre order \(M\) and use the
scaling rule \(\alpha=4M/(5L)\). The minimal order \(M_{\min}\) is recorded
when the approximation error falls below the prescribed tolerance.

Figure~\ref{fig:residual_phase_complexity} shows the dependence of
\(M_{\min}\) on the accumulated residual phase. The results indicate an
approximately linear growth of the required Laguerre order with respect to
\(\Lambda\). This is consistent with the interpretation that, after
modulation, the Laguerre expansion only needs to resolve the residual
oscillation in the envelope, rather than the original carrier frequency.

The important observation is that \(M_{\min}\) is governed by the residual
phase \(|\omega-\kappa|L\), not by the full carrier phase \(|\omega|L\).
Therefore, the frequency detection step does not need to recover the
carrier frequency to the final approximation accuracy. It is sufficient to
identify an effective modulation center so that the remaining residual
phase is moderate. The residual oscillation can then be absorbed by a
moderate-order Laguerre envelope approximation.

\begin{figure}[htbp]
\centering
\includegraphics[width=0.65\textwidth]{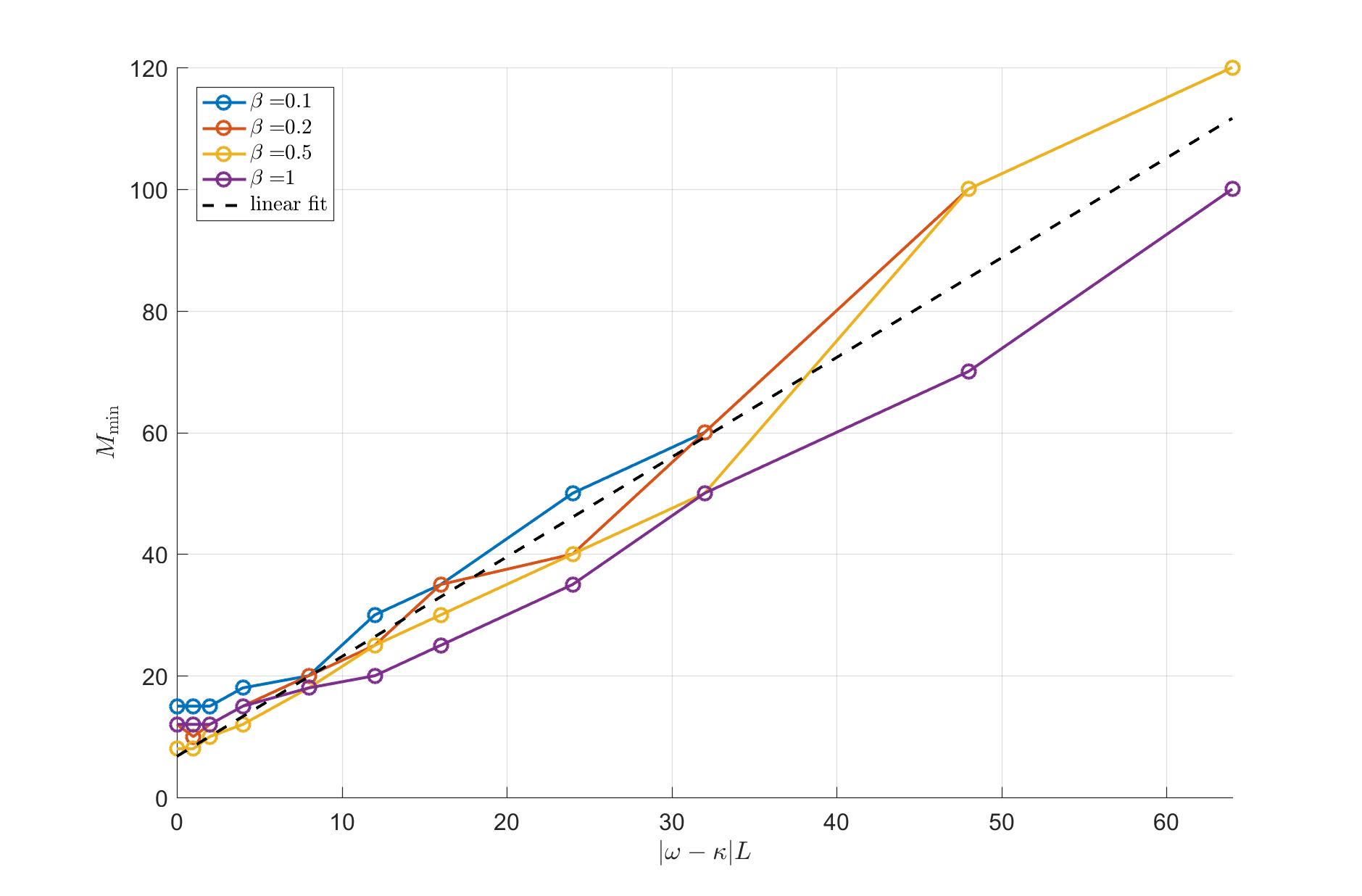}
\caption{
Minimal Laguerre order \(M_{\min}\) required to reach the prescribed
accuracy as a function of the accumulated residual phase
\(\Lambda=|\omega-\kappa|L\). The scaling parameter is fixed by
\(\alpha=4M/(5L)\). The dashed line indicates an approximately linear
empirical trend.
}
\label{fig:residual_phase_complexity}
\end{figure}

\subsection{Choice of the sampling ratio}

After fixing the scaling rule
\[
    \alpha=\frac{4M}{5L},
\]
and selecting the Laguerre order according to the residual phase
\[
    \Lambda=|\omega-\kappa|L,
\]
the remaining numerical parameter is the number of sampling points used in
the tail least-squares problem.

For a Laguerre expansion of order \(M\), the number of unknown coefficients
for a single unmodulated envelope is \(M+1\). We define the sampling ratio by
\[
    \gamma=\frac{m}{M+1},
\]
where \(m\) denotes the number of sampling points on the tail interval. In
the modulated multi-frequency case, the same idea is applied to the full
least-squares system by choosing the number of samples proportional to the
number of unknowns. The role of \(\gamma\) is analogous to the oversampling
parameter in Fourier extension methods: a small value may lead to an
unstable or under-resolved discrete least-squares problem, whereas excessive
oversampling increases the computational cost without substantially
improving the final accuracy.

To determine a suitable value of \(\gamma\), we fix the scaling rule
\(\alpha=4M/(5L)\) and test several representative residual phase values
\[
    \Lambda=|\omega-\kappa|L .
\]
For each value of \(\Lambda\), the Laguerre order \(M\) is chosen according
to the residual-phase test in the previous subsection, and only the sampling
ratio is varied.

Figure~\ref{fig:sampling_ratio} shows the influence of \(\gamma\) on the
modulated Laguerre approximation error. The error decreases rapidly when
\(\gamma\) increases from values close to one, especially for larger
residual phases. After \(\gamma\approx 3\)--\(4\), the improvement becomes
much less significant and the curves enter a nearly saturated regime. This
indicates that a moderate amount of oversampling is sufficient for the
TSVD-regularized tail least-squares approximation.

Therefore, in the following experiments we use
\[
    \boxed{\gamma=4}
\]
as the default sampling ratio. This choice provides a stable discrete
least-squares system while avoiding unnecessary oversampling in the tail
approximation.

\begin{figure}[htbp]
\centering
\includegraphics[width=0.65\textwidth]{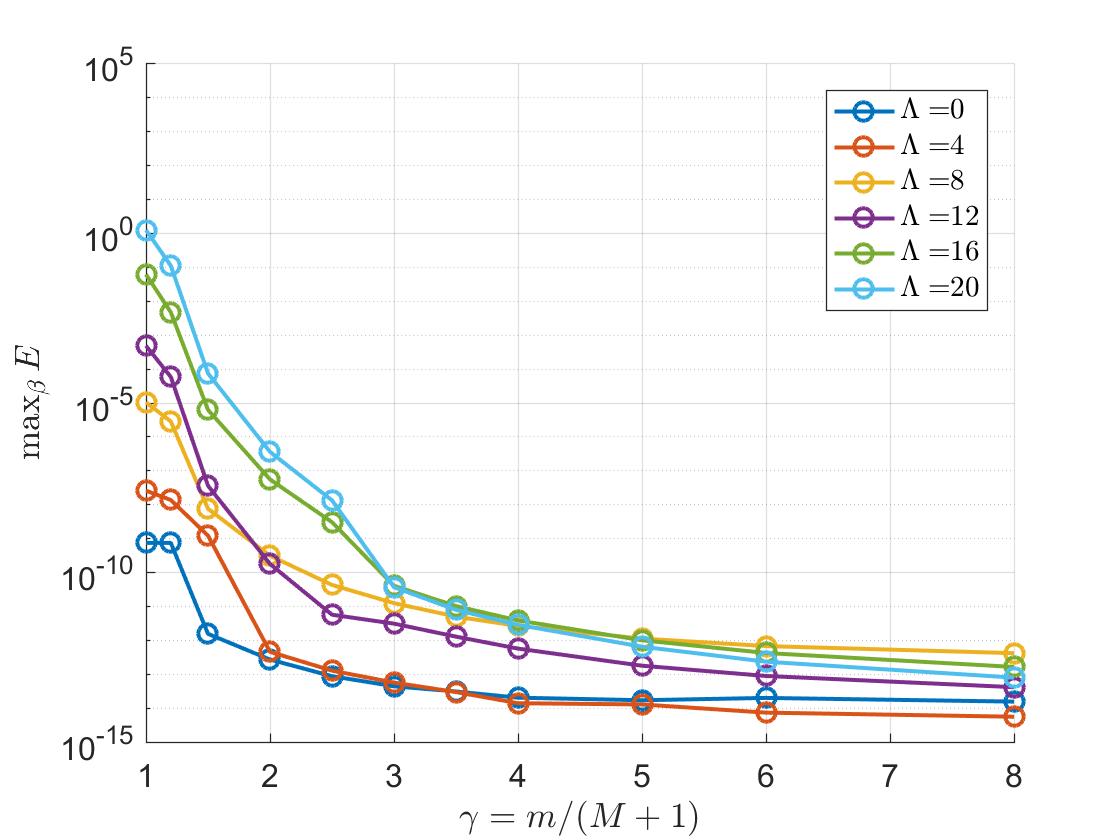}
\caption{
Influence of the sampling ratio \(\gamma=m/(M+1)\) on the modulated
Laguerre approximation error. The scaling parameter is fixed by
\(\alpha=4M/(5L)\), and the Laguerre order is selected according to the
residual phase \(\Lambda=|\omega-\kappa|L\). The results show that
\(\gamma=4\) gives a stable oversampling level for the subsequent
computations.
}
\label{fig:sampling_ratio}
\end{figure}
\subsection{Computational complexity}
\label{subsec:complexity}

We briefly discuss the computational cost of the proposed core--tail
approximation. The total cost consists of four parts: the local Fourier
extension approximation in the core, the possible internal edge-detection
preprocessing in the core, the modulated Laguerre approximation on the tails,
and the tail-frequency detection step.

On the core interval, suppose that the interval is divided into \(K\)
subintervals. On each subinterval, we use \(m_c\) sampling points and
\(n_c=2N+1\) local Fourier extension basis functions. The least-squares
problem on one subinterval has size \(m_c\times n_c\). If it is solved by a
truncated singular value decomposition, the cost on one subinterval is
\[
    O(m_c n_c^2+n_c^3).
\]
Since the same local parameters are used on all subintervals in the present
implementation, the total cost in the core is
\begin{equation}
    O\bigl(K(m_c n_c^2+n_c^3)\bigr).
    \label{eq:complexity_core}
\end{equation}
In practice, \(n_c\) is fixed and small, while \(K\) increases with the local
oscillation or structural complexity of the function. Therefore the core
fitting cost grows essentially linearly with the number of local
subintervals.

The internal edge-detection step, when used, is also based on local LFE
systems. If \(K_e\) local windows are tested and each uses the same order as
the core approximation, the cost is of the form
\[
    O\bigl(K_e(m_e n_e^2+n_e^3)\bigr),
\]
where \(m_e\) and \(n_e\) denote the number of samples and basis functions in
one detection window. After a candidate window is found, the subgrid
localization uses only a small number of one-sided LFE solves. Hence this
additional cost is a local preprocessing cost and has the same scaling type
as the core LFE approximation. It does not change the leading complexity
order when the local window size is fixed.

For the tail approximation, assume that \(q\) modulation frequencies are used
on one tail and that \(M+1\) Laguerre functions are assigned to each
frequency. Since both cosine- and sine-modulated blocks are included, the
number of tail unknowns is approximately
\[
    n_t = 2q(M+1).
\]
Let \(m_t=\gamma_t n_t\) be the number of tail sampling points. A direct
truncated SVD solve on one tail then costs
\[
    O(m_t n_t^2+n_t^3).
\]
Since there are two tails, the total tail fitting cost is
\begin{equation}
    O\bigl(2(m_t n_t^2+n_t^3)\bigr).
    \label{eq:complexity_tail}
\end{equation}
This term depends strongly on the Laguerre order \(M\), because
\(n_t=2q(M+1)\). Thus reducing \(M\) has a substantial effect on the tail
cost. This is one motivation for using the scaling rule
\[
    \alpha=\frac{4M}{5L},
\]
which allows moderate tail orders, typically \(M=25\)--\(30\), to be used in
the numerical experiments. With small \(q\) and moderate \(M\), the tail
systems remain much smaller than a global high-dimensional approximation on
the whole real line.

The tail-frequency detection step is performed only near the two
core--tail interfaces. Let \(m_{\rm det}\) be the number of detection samples
and let \(n_{\rm det}=2N_{\rm det}+1\) be the number of Fourier detection
modes. A direct SVD-based detection has cost
\[
    O(m_{\rm det}n_{\rm det}^2+n_{\rm det}^3)
\]
on one detection window. This step is used only to identify a small number of
dominant modulation centers. It is therefore sufficient to use a coarse
detection grid. In the numerical experiments, we use
\[
    m_{\rm det}=201,\qquad T_{\rm det}=4,\qquad
    \kappa_{\max}=80,\qquad \varepsilon_{\rm det}=10^{-8},
\]
which already gives reliable modulation frequencies for the subsequent tail
approximation. Hence the frequency detection cost is a preprocessing cost and
does not grow with the number of core subintervals.

Combining the above estimates, the overall cost can be summarized as
\begin{equation}
\begin{aligned}
    {\rm Cost}
    =&\; O\bigl(K(m_c n_c^2+n_c^3)\bigr)
    + O\bigl(K_e(m_e n_e^2+n_e^3)\bigr) \\
    &\quad
    + O\bigl(2(m_t n_t^2+n_t^3)\bigr)
    + O\bigl(2(m_{\rm det}n_{\rm det}^2+n_{\rm det}^3)\bigr).
\end{aligned}
\label{eq:overall_complexity}
\end{equation}
When the local basis sizes \(n_c\), \(n_e\), \(n_t\), and \(n_{\rm det}\) are
fixed, and when the number of tail modulation centers \(q\) remains small,
the dominant growth with respect to the core refinement is essentially
linear in \(K\). This is the main computational advantage of the localized
construction: refinement is introduced only where needed, while the tail
behavior is captured by a small number of modulated Laguerre blocks rather
than by a large global basis.

\section{Numerical experiments}
\label{sec:numerical}

In this section, we test the proposed core--tail Fourier--Laguerre
approximation on representative functions and a model problem on the real
line. The finite core is approximated by local Fourier extension (LFE), and
the two tails are approximated by modulated Laguerre frames. Unless otherwise
specified, the Laguerre scaling is chosen according to
\[
    \alpha=\frac{4M}{5L},
\]
where \(L\) is the computational tail length and \(M\) is the Laguerre order.
The sampling ratio in the tail least-squares problem is fixed as
\[
    \gamma=4 .
\]
All least-squares systems are solved by TSVD, with truncation threshold
\(10^{-13}\).

\subsection{Tail frequency detection}
\label{subsec:tail-frequency-detection}

The modulated Laguerre approximation requires modulation frequencies in the
tails. These frequencies are used as modulation centers rather than exact
carrier frequencies. Thus the frequency detection step is only required to
identify the dominant oscillatory components at a coarse level. Any moderate
frequency mismatch becomes a residual oscillation in the Laguerre envelope.

On a detection window \(x\in[0,L]\), we introduce
\[
    t=\frac{2(x-L/2)}{L},\qquad t\in[-1,1],
\]
and construct a local Fourier extension system with modes
\[
    \exp\left(\frac{i k\pi t}{T_{\rm det}}\right),
    \qquad k=-N_{\rm det},\ldots,N_{\rm det}.
\]
The physical frequency associated with mode \(k\) is
\[
    \kappa_k=\frac{2\pi k}{T_{\rm det}L}.
\]
After solving the least-squares problem by TSVD, we use the coefficient
indicator
\[
    I(k)=|c_k|+|c_{-k}|,\qquad k>0,
\]
to identify dominant positive-frequency components. For multi-frequency
tails, well-separated local maxima of \(I(k)\) are retained as candidate
modulation centers.

In the tests below, we use a reduced detection system with
\[
    L=8,\qquad
    m_{\rm det}=201,\qquad
    T_{\rm det}=4,\qquad
    \kappa_{\max}=80,\qquad
    \varepsilon_{\rm det}=10^{-8}.
\]
The detection order is chosen as
\[
    N_{\rm det}
    =
    \left\lceil
    \frac{\kappa_{\max}T_{\rm det}L}{2\pi}
    \right\rceil ,
\]
so the detection matrix has size \(201\times(2N_{\rm det}+1)\). This system
is small compared with the final approximation problems, but it is sufficient
for selecting effective modulation centers.

We consider the following six tail functions:
\[
\begin{aligned}
g_1(x)&=e^{-0.2x}\cos(50x),\\
g_2(x)&=e^{-0.05x^2}\cos(50x),\\
g_3(x)&=(1+x)^{-2}\cos(50x),\\
g_4(x)&=e^{-0.2x}\cos(50x+0.5\sin x),\\
g_5(x)&=e^{-0.2x}\left[\cos(40x)+0.5\cos(60x)\right],\\
g_6(x)&=e^{-0.2x}\left[\cos(40x)+0.5\cos(55x)+0.3\cos(70x)\right].
\end{aligned}
\]
The first three functions have the same carrier frequency but different
decay envelopes. The fourth function contains a weakly varying phase, and
the last two functions are multi-frequency examples.

Figures~\ref{fig:tail_frequency_single} and
\ref{fig:tail_frequency_multi} show the normalized indicators. For
\(g_1\)--\(g_3\), the dominant peaks are located near the reference carrier
frequency \(\omega=50\). For \(g_4\), the weak phase modulation broadens the
peak slightly, but the indicator still gives an effective modulation center.
For \(g_5\) and \(g_6\), the two and three dominant frequency components are
clearly separated.

\begin{figure}[htbp]
\centering
\includegraphics[width=0.70\textwidth]{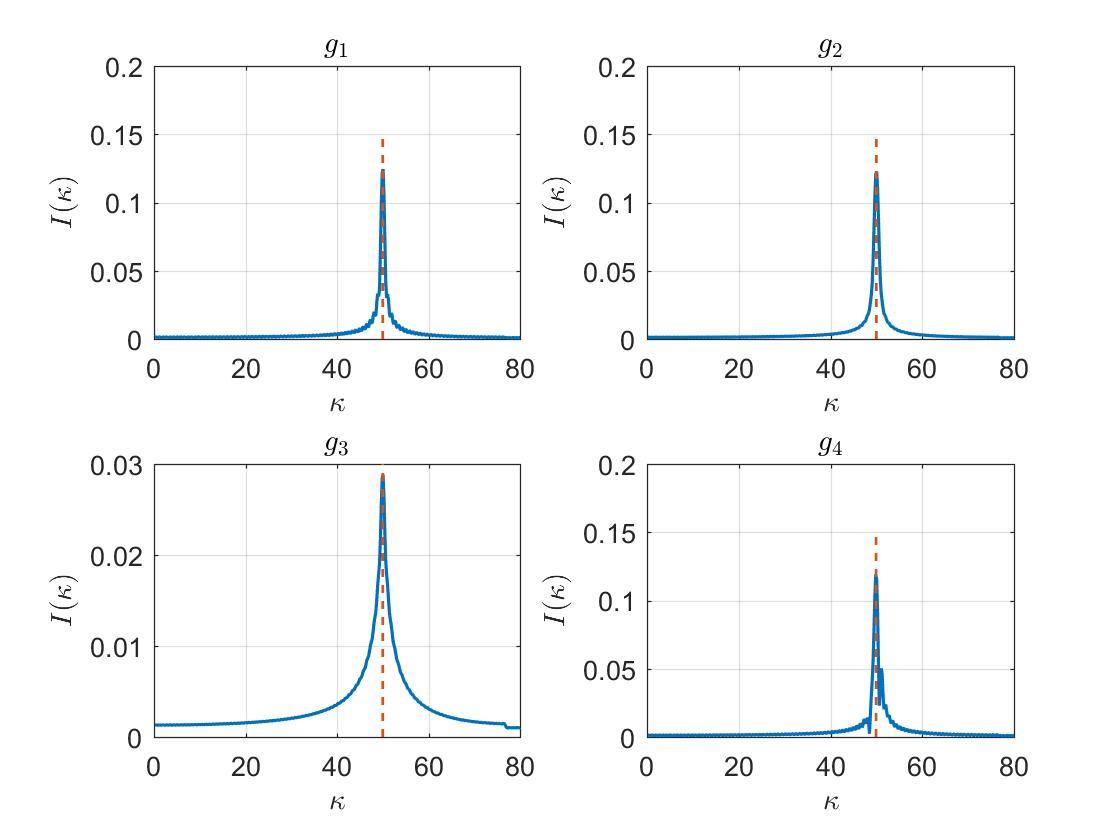}
\caption{Normalized LFE coefficient indicators for single-frequency and
weakly modulated tail functions. The dashed vertical lines indicate the
reference carrier frequencies.}
\label{fig:tail_frequency_single}
\end{figure}

\begin{figure}[htbp]
\centering
\includegraphics[width=0.70\textwidth]{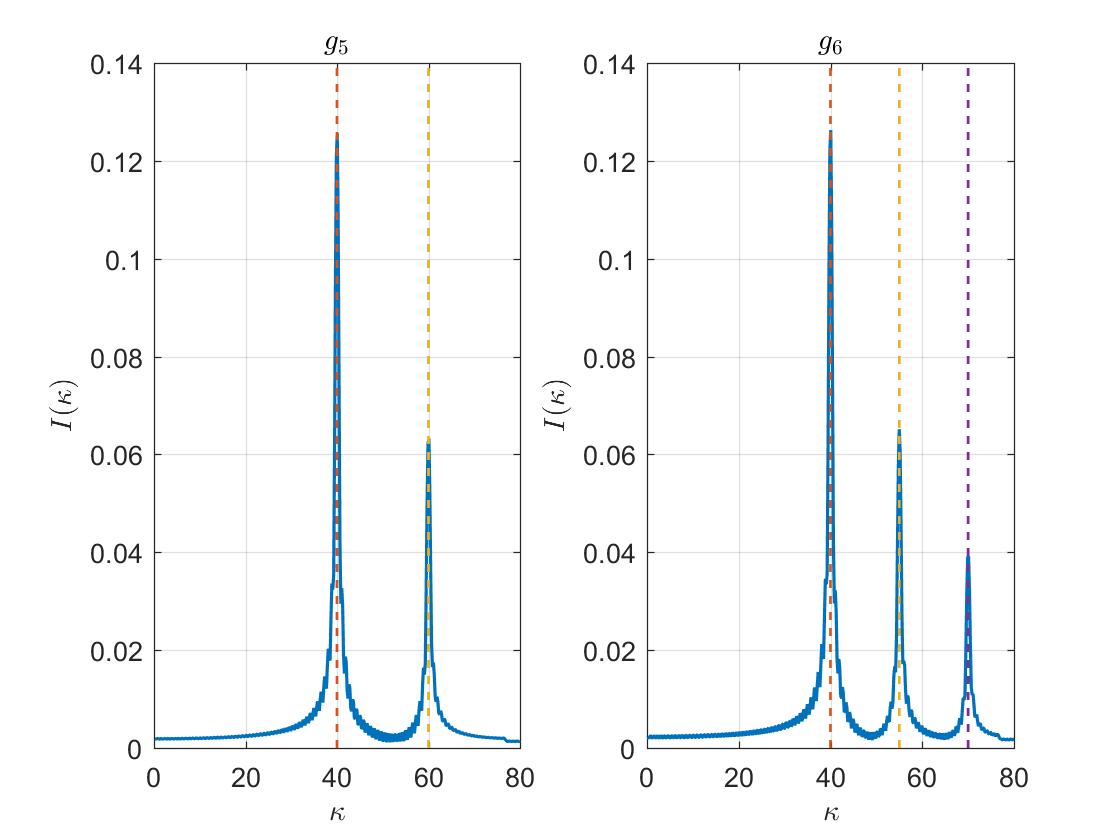}
\caption{Normalized LFE coefficient indicators for multi-frequency tail
functions. The detected peaks are used as modulation centers in the
modulated Laguerre tail approximation.}
\label{fig:tail_frequency_multi}
\end{figure}

The detected frequencies are summarized in
Table~\ref{tab:tail_frequency_detection}. We report the detection error by
the accumulated residual phase
\[
    |\omega_j-\kappa_j|L,
\]
because this is the quantity that determines the residual oscillation left
in the Laguerre envelope after modulation.

\begin{table}[htbp]
\centering
\caption{Tail frequency detection results with the reduced detection system.}
\label{tab:tail_frequency_detection}
\small
\setlength{\tabcolsep}{5pt}
\begin{tabular*}{\textwidth}{@{\extracolsep\fill}lccccccccc}
\toprule
 & \(g_1\) & \(g_2\) & \(g_3\) & \(g_4\)
 & \(g_{5,1}\) & \(g_{5,2}\)
 & \(g_{6,1}\) & \(g_{6,2}\) & \(g_{6,3}\) \\
\midrule
Reference frequency
& \(50.0000\) & \(50.0000\) & \(50.0000\) & \(50.0796\)
& \(40.0000\) & \(60.0000\) & \(40.0000\) & \(55.0000\) & \(70.0000\) \\
Detected frequency
& \(49.9991\) & \(49.9991\) & \(49.9966\) & \(49.9318\)
& \(39.9987\) & \(60.0035\) & \(39.9937\) & \(55.0064\) & \(70.0047\) \\
\(|\omega_j-\kappa_j|L\)
& \(0.0069\) & \(0.0070\) & \(0.0275\) & \(1.1822\)
& \(0.0104\) & \(0.0281\) & \(0.0504\) & \(0.0512\) & \(0.0376\) \\
\bottomrule
\end{tabular*}
\end{table}

For the strictly single-frequency cases \(g_1\)--\(g_3\), the accumulated
residual phase is below \(3\times10^{-2}\). For the multi-frequency cases
\(g_5\) and \(g_6\), all dominant components are detected with small residual
phases. The larger value for \(g_4\) is expected, since this function has a
weakly varying instantaneous frequency. In that case, the detected frequency
should be interpreted as an effective modulation center rather than an exact
carrier frequency. These results support the use of a coarse LFE-based
frequency detector in the construction of the modulated Laguerre tails.

\subsection{Comparison with Hermite approximation}
\label{subsec:hermite-comparison}

We next compare the proposed adaptive core--tail approximation with a
fixed-scaling Hermite approximation on $[-10,10]$. The purpose of this
comparison is not to show that one basis is universally superior, but to
demonstrate their different advantages. Hermite functions provide an efficient
global representation for smooth functions with Gaussian-type decay, whereas
the proposed localized method is designed for functions containing oscillatory
tails, multiple frequencies, localized structures, and internal low-regularity
points.

The Hermite approximation uses the scaled basis
\[
    \psi_n^\lambda(x)=\sqrt{\lambda}\psi_n(\lambda x),
    \qquad n=0,1,\ldots,N_H,
\]
where the scaling parameter is fixed as
\[
    \lambda=\sqrt{2}.
\]
This choice avoids case-dependent optimization and provides a consistent
baseline. For the proposed method, the local Fourier extension parameters are
fixed in the core, while the tail Laguerre scaling follows
\[
    \alpha=\frac{4M}{5L}.
\]
The tail modulation frequencies are determined by the LFE-based frequency
indicator introduced previously.

We first consider six representative functions:
\[
\begin{aligned}
f_1(x)&=e^{-x^2}(1+0.2\cos 2x),\\
f_2(x)&=e^{-0.18\sqrt{1+x^2}}\cos(30x),\\
f_3(x)&=e^{-0.18\sqrt{1+x^2}}
       \left[\cos(30x)+0.45\cos(48x)\right],\\
f_4(x)&=e^{-0.14\sqrt{1+x^2}}
       \left[\cos(28x)+0.45\cos(34x)+0.25\cos(44x)\right],\\
f_5(x)&=e^{-x^2}+0.3e^{-15x^2}\cos(80x),\\
f_6(x)&=e^{-x^2}\operatorname{erf}(80x).
\end{aligned}
\]

The first function is globally smooth and Gaussian-type. The next three
functions contain oscillatory tails with increasing frequency complexity,
while the last two contain localized structures in the core region.

Figure~\ref{fig:hermite_comparison} reports the maximum approximation error
with respect to the number of approximation degrees of freedom. Here the
horizontal axis is defined consistently as
\[
    N_{\rm app}
    =
    N_{\rm core}^{\rm DOF}
    +
    N_{\rm tail}^{\rm DOF},
\]
where only basis coefficients are counted; least-squares sampling points are
not included.

For the smooth Gaussian-type function $f_1$, the Hermite approximation shows
excellent efficiency, which is consistent with the global nature of the
Hermite basis. For oscillatory tail functions $f_2$--$f_4$, the proposed method
benefits from frequency modulation and local resolution, avoiding the need for
a large global expansion. The advantage becomes more evident for $f_5$ and
$f_6$, where localized oscillations or sharp core structures significantly
reduce the efficiency of a global Hermite representation.

\begin{figure}[htbp]
\centering
\includegraphics[width=0.95\textwidth]{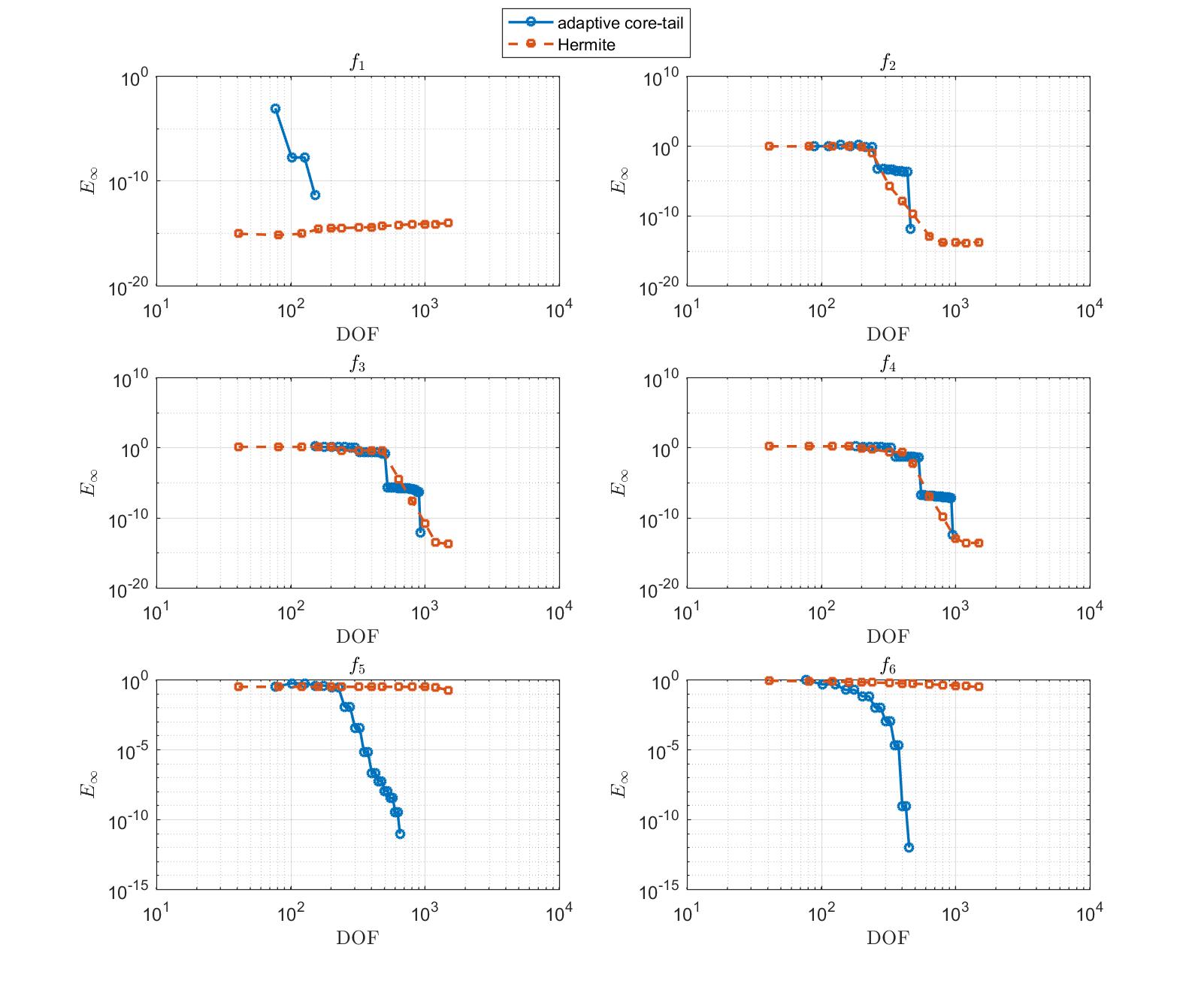}
\caption{
Maximum error $E_\infty$ versus approximation degrees of freedom for the
adaptive core--tail approximation and the fixed-scaling Hermite approximation.
}
\label{fig:hermite_comparison}
\end{figure}

We further consider a continuous function with an internal derivative jump:
\[
\begin{aligned}
f_{\rm dj}(x)
&=
e^{-0.18\sqrt{1+x^2}}
\left(\cos(30x)+0.45\cos(48x)\right)
+0.25e^{-3x^2}|x-\xi|,
\\
&\qquad \xi=\frac{\sqrt2}{4}.
\end{aligned}
\]

The first term describes two oscillatory decaying tail components, while the
second term introduces a localized derivative discontinuity in the core.
The computational domain is $[-10,10]$, and the core interval is $[-3,3]$.
The LFE parameters are $N=12$ and $T=6$, while the Laguerre tail uses
$M=30$ and the detected modulation frequencies $30$ and $48$.

The internal edge is detected through the coefficient-energy indicator
\[
    \eta_k=\|c_k^\epsilon\|_2 .
\]
The detected location is
\[
\xi=0.353553390593274,
\qquad
\widehat{\xi}=0.353553390591261,
\]
with localization error
\[
|\widehat{\xi}-\xi|
=2.013\times10^{-12}.
\]
Thus the detected edge is sufficiently accurate for subgrid partition
alignment.

Figure~\ref{fig:derivative_jump_detection} illustrates the edge detection
process. The coefficient-energy indicator produces a clear peak around the
low-regularity point, and the one-sided LFE models provide a subgrid estimate
of the edge location.

\begin{figure}[htbp]
\centering
\includegraphics[width=0.7\textwidth]{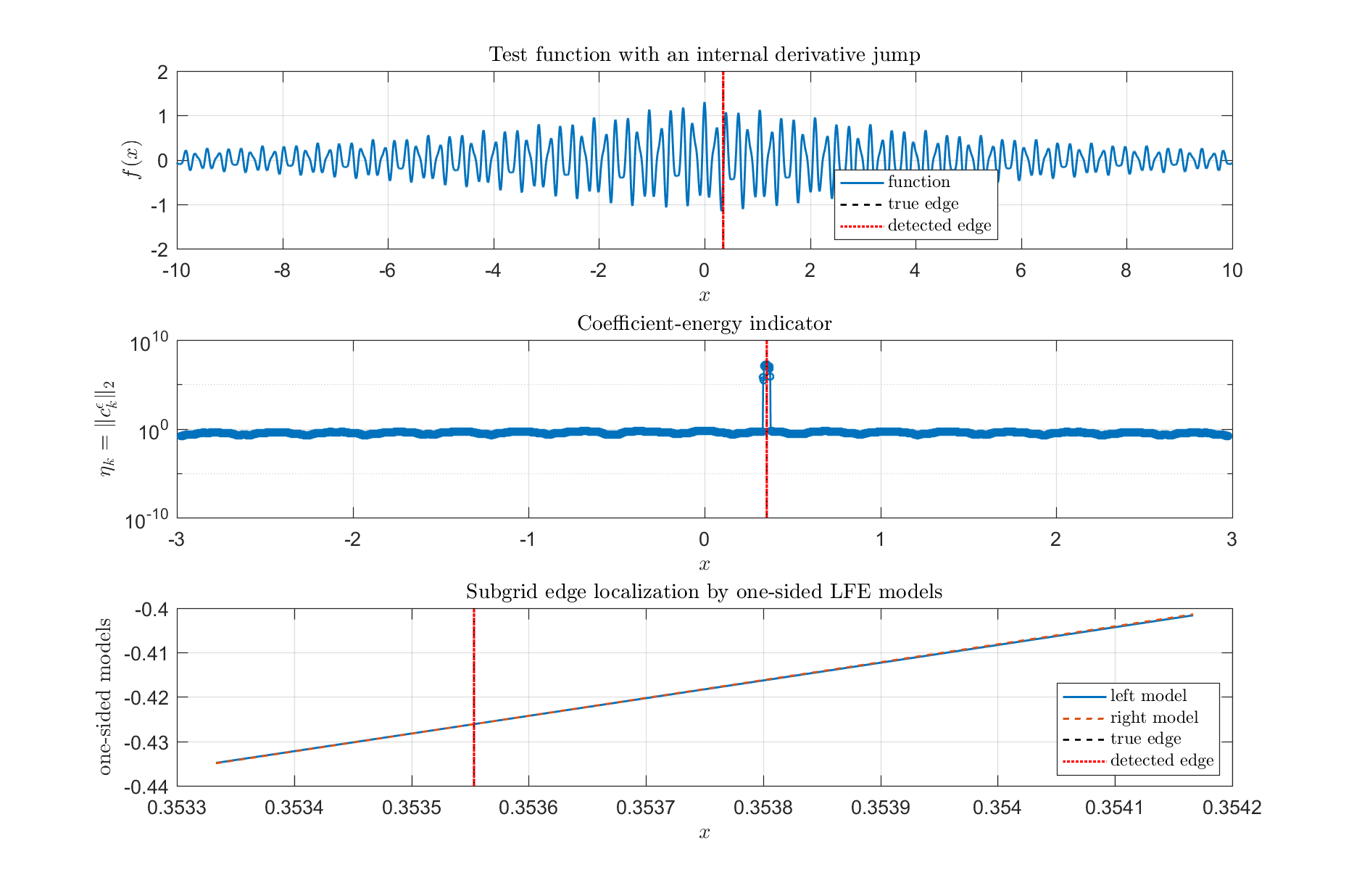}
\caption{
Detection and subgrid localization of an internal derivative jump using the
LFE coefficient-energy indicator and one-sided local models.
}
\label{fig:derivative_jump_detection}
\end{figure}

Table~\ref{tab:derivative_jump_coretail} compares the aligned and unaligned
core partitions. The results show that without edge alignment, the local
Fourier approximation is affected by the derivative discontinuity and the
error stagnates around $10^{-4}$. After inserting the detected edge into the
partition, spectral accuracy is recovered and the error decreases to the
roundoff level.

\begin{table}[htbp]
\centering
\caption{
Errors for the derivative-jump example. The node number denotes approximation
degrees of freedom.
}
\label{tab:derivative_jump_coretail}
\begin{tabular}{ccccc}
\toprule
$K$ & DOF
& $E_\infty$ (aligned)
& $E_2$ (aligned)
& $E_\infty$ (without alignment)
\\
\midrule
16 & 728
& $1.806\times10^{-6}$
& $1.486\times10^{-7}$
& $5.384\times10^{-4}$
\\
24 & 928
& $3.219\times10^{-12}$
& $2.314\times10^{-13}$
& $6.493\times10^{-4}$
\\
32 & 1128
& $4.594\times10^{-13}$
& $7.501\times10^{-15}$
& $3.093\times10^{-4}$
\\
\bottomrule
\end{tabular}
\end{table}

Figure~\ref{fig:derivative_jump_pointwise} shows the pointwise error
distribution. The edge-aligned core--tail approximation remains close to
machine precision except in a very small neighborhood of the detected edge.
In contrast, the Hermite approximation is globally influenced by the localized
nonsmooth structure. Even with $N_H=1500$, the Hermite approximation gives
\[
E_\infty=1.347\times10^{-3},
\qquad
E_2=3.549\times10^{-5}.
\]

\begin{figure}[htbp]
\centering
\includegraphics[width=0.65\textwidth]{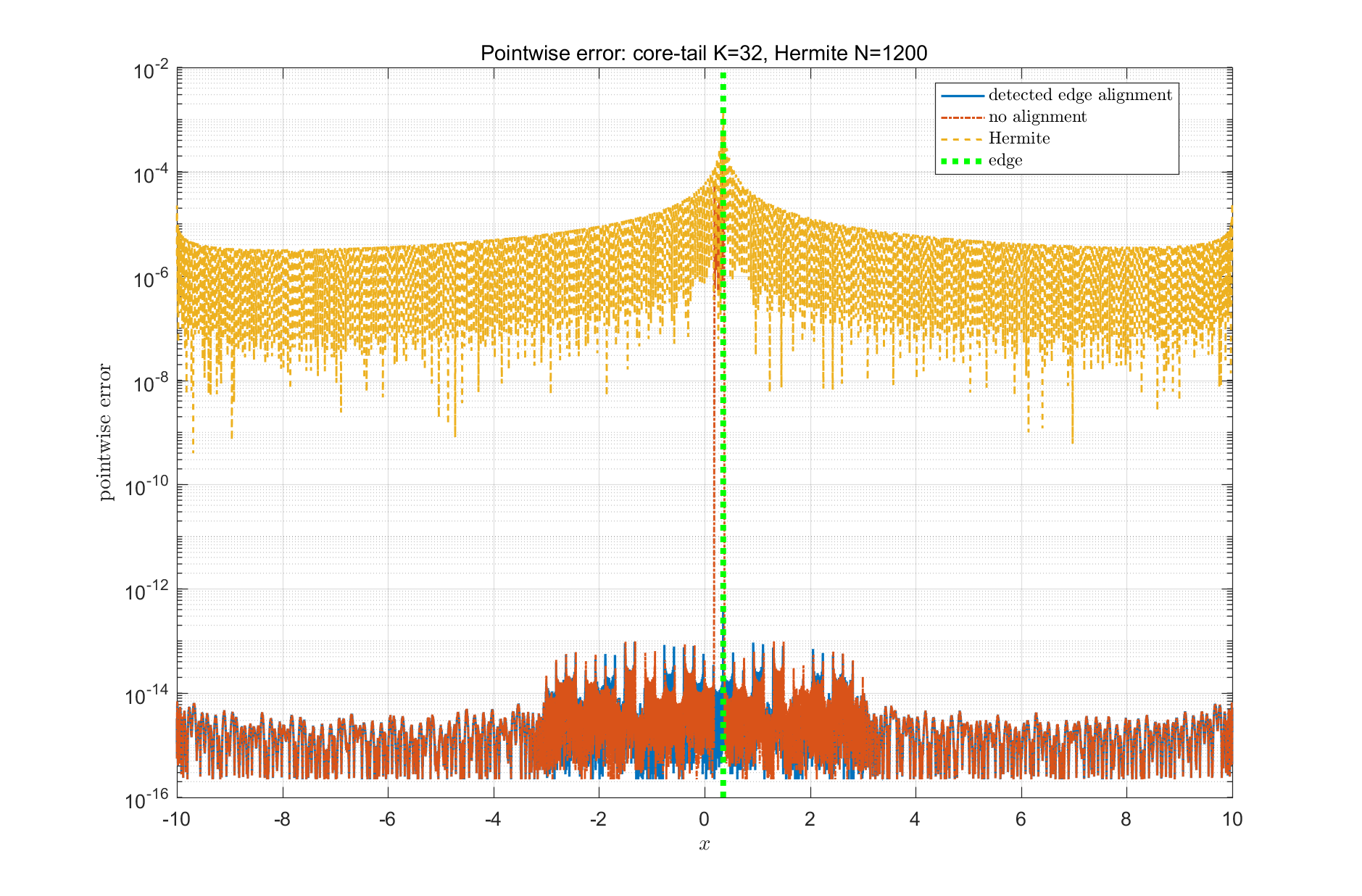}
\caption{
Pointwise errors for the derivative-jump example. The core--tail result uses
$K=32$, while the Hermite approximation uses $N_H=1200$.
}
\label{fig:derivative_jump_pointwise}
\end{figure}

Finally, we compare computational efficiency. The runtime comparison is based
on the same approximation-degree-of-freedom definition used above. The
frequency detection step is included as a preprocessing cost.

Figure~\ref{fig:runtime_comparison} shows that the localized construction has
a slower growth rate with respect to the approximation size. This is because
the proposed method solves several small local systems rather than repeatedly
forming one large global approximation system.

\begin{figure}[htbp]
\centering
\includegraphics[width=0.68\textwidth]{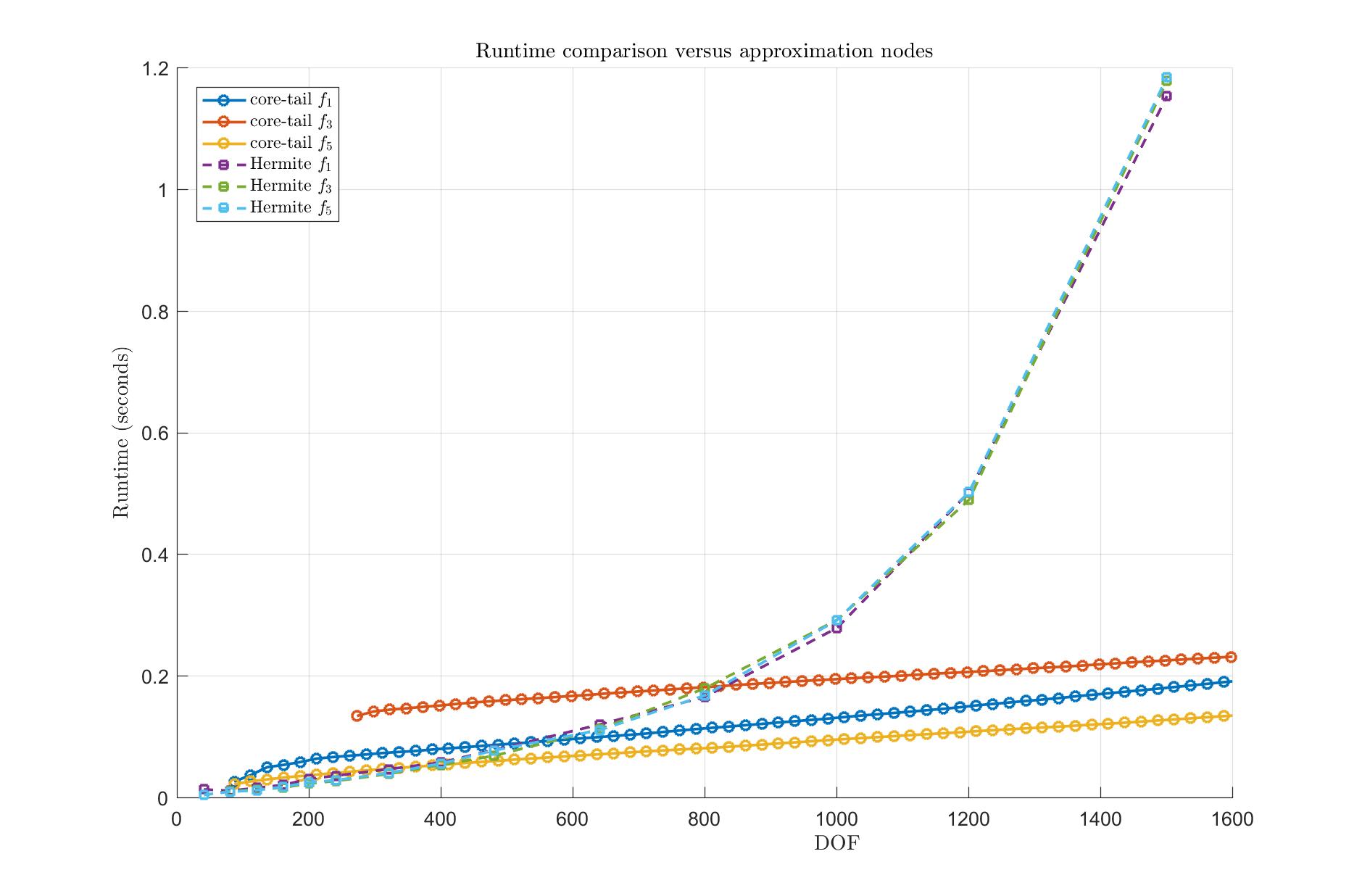}
\caption{
Runtime comparison between the adaptive core--tail approximation and the
fixed-scaling Hermite approximation using approximation degrees of freedom as
the complexity measure.
}
\label{fig:runtime_comparison}
\end{figure}

Overall, the comparison demonstrates that Hermite approximation remains an
excellent choice for globally smooth Gaussian-type functions, while the
proposed core--tail Fourier--Laguerre framework provides clear advantages for
functions containing oscillatory tails, multiple frequencies, localized
structures, and internal derivative discontinuities.

\subsection{A decaying model problem on the real line}
\label{subsec:model-problem}

We finally consider a decaying model problem on the real line,
\begin{equation}
    -u''(x)+\gamma u(x)=f(x),\qquad x\in\mathbb{R},
    \qquad
    \lim_{|x|\rightarrow\infty}u(x)=0,
    \label{eq:model_problem}
\end{equation}
where $\gamma>0$. This example demonstrates that the proposed core--tail
representation can be incorporated into differential operators while
preserving the decay property and the interface continuity.

We set $\gamma=4$ and prescribe the exact solution
\begin{equation}
\begin{aligned}
u_{\rm ex}(x)
=&\;e^{-0.45\sqrt{1+x^2}}\cos(20x)
+0.35e^{-0.45\sqrt{1+x^2}}\cos(32x)
\\
&\quad
+0.25e^{-3x^2}\cos(12x).
\end{aligned}
\label{eq:model_exact_solution}
\end{equation}
The right-hand side is generated analytically by
\begin{equation}
    f(x)=-u_{\rm ex}''(x)+\gamma u_{\rm ex}(x).
\end{equation}
Therefore, the reference solution does not involve any numerical
quadrature or additional approximation error.

The computational domain is truncated to $[-12,12]$ with the core interval
$[-3,3]$. The core approximation uses LFE with
\[
    N=12,\qquad T=6,
\]
while the two tails are approximated by modulated Laguerre frames with
\[
    M=40,\qquad
    \kappa_L=\kappa_R=\{20,32\}.
\]

The numerical solution is represented by the piecewise expansion
\[
u_N(x)=
\begin{cases}
u_L(x),&x<-3,\\
u_j(x),&x\in I_j,\quad j=1,\ldots,K,\\
u_R(x),&x>3.
\end{cases}
\]

Since all local basis functions are explicitly differentiable, the operator
$-\partial_{xx}+\gamma$ is applied directly to the core Fourier extension
basis and the modulated Laguerre basis. At the interfaces, the continuity
conditions
\begin{equation}
    [u_N]=0,\qquad [u_N']=0
    \label{eq:model_interface_conditions}
\end{equation}
are imposed. These constraints are incorporated through a nullspace
reduction. Specifically, if $Ec=0$ denotes the interface constraints, we write
$c=Zy$, where the columns of $Z$ span the nullspace of $E$, and solve the
resulting reduced least-squares system for $y$.

Table~\ref{tab:model_problem} reports the numerical results. Here the reported
nodes correspond to the approximation degrees of freedom (DOF), namely the
number of expansion coefficients. This definition is consistent with the
complexity comparison against global spectral approximations.

\begin{table}[htbp]
\centering
\caption{Errors for the decaying model problem with constrained core--tail
collocation.}
\label{tab:model_problem}
\begin{tabular}{cccccc}
\toprule
$K$
& DOF
& Reduced DOF
& $E_\infty(u)$
& $E_2(u)$
& $\max\{|[u]|,|[u']|\}$
\\
\midrule
 2  & 378  & 372  & $9.417\times10^{-1}$
& $1.666\times10^{-1}$ & $2.27\times10^{-3}$\\
 4  & 428  & 418  & $3.888\times10^{-1}$
& $7.353\times10^{-2}$ & $5.68\times10^{-4}$\\
 6  & 478  & 464  & $2.121\times10^{-1}$
& $4.595\times10^{-2}$ & $4.90\times10^{-4}$\\
 8  & 528  & 510  & $2.134\times10^{-3}$
& $4.474\times10^{-4}$ & $5.60\times10^{-6}$\\
10  & 578  & 556  & $5.419\times10^{-5}$
& $1.331\times10^{-5}$ & $1.65\times10^{-7}$\\
12  & 628  & 602  & $4.615\times10^{-8}$
& $9.154\times10^{-9}$ & $5.66\times10^{-9}$\\
16  & 728  & 694  & $1.456\times10^{-11}$
& $4.764\times10^{-12}$ & $2.60\times10^{-13}$\\
20  & 828  & 786  & $6.368\times10^{-13}$
& $1.722\times10^{-13}$ & $2.98\times10^{-14}$\\
24  & 928  & 878  & $1.299\times10^{-13}$
& $2.572\times10^{-14}$ & $4.87\times10^{-14}$\\
32  &1128  &1062  & $4.692\times10^{-13}$
& $1.182\times10^{-13}$ & $4.10\times10^{-14}$\\
\bottomrule
\end{tabular}
\end{table}

The convergence history is shown in Figure~\ref{fig:model_error_nodes}.
The error decreases rapidly as the number of local core subintervals increases.
The maximum error reaches the level of $10^{-13}$ around $K=20$, after which
the improvement becomes limited by TSVD truncation and floating-point
round-off effects.

\begin{figure}[htbp]
\centering
\includegraphics[width=0.68\textwidth]{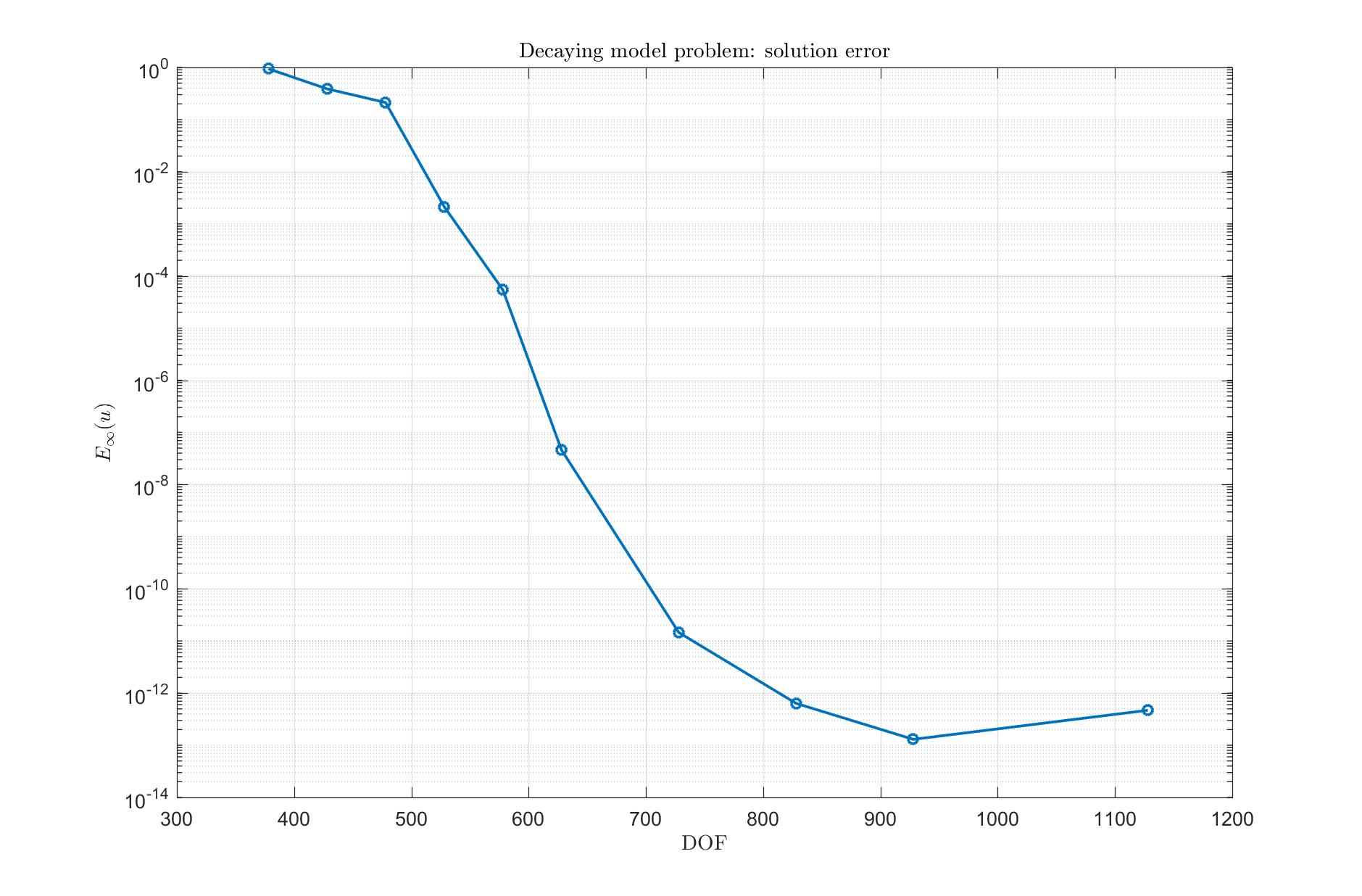}
\caption{Maximum solution error versus approximation degrees of freedom for
the decaying model problem.}
\label{fig:model_error_nodes}
\end{figure}

Figure~\ref{fig:model_pointwise} shows the pointwise error distribution for the
representative case $K=20$ with 828 approximation degrees of freedom. The
error remains below $10^{-12}$ over the whole computational domain, indicating
that the constrained core--tail formulation preserves both the decay behavior
and the interface continuity.

\begin{figure}[htbp]
\centering
\includegraphics[width=0.68\textwidth]{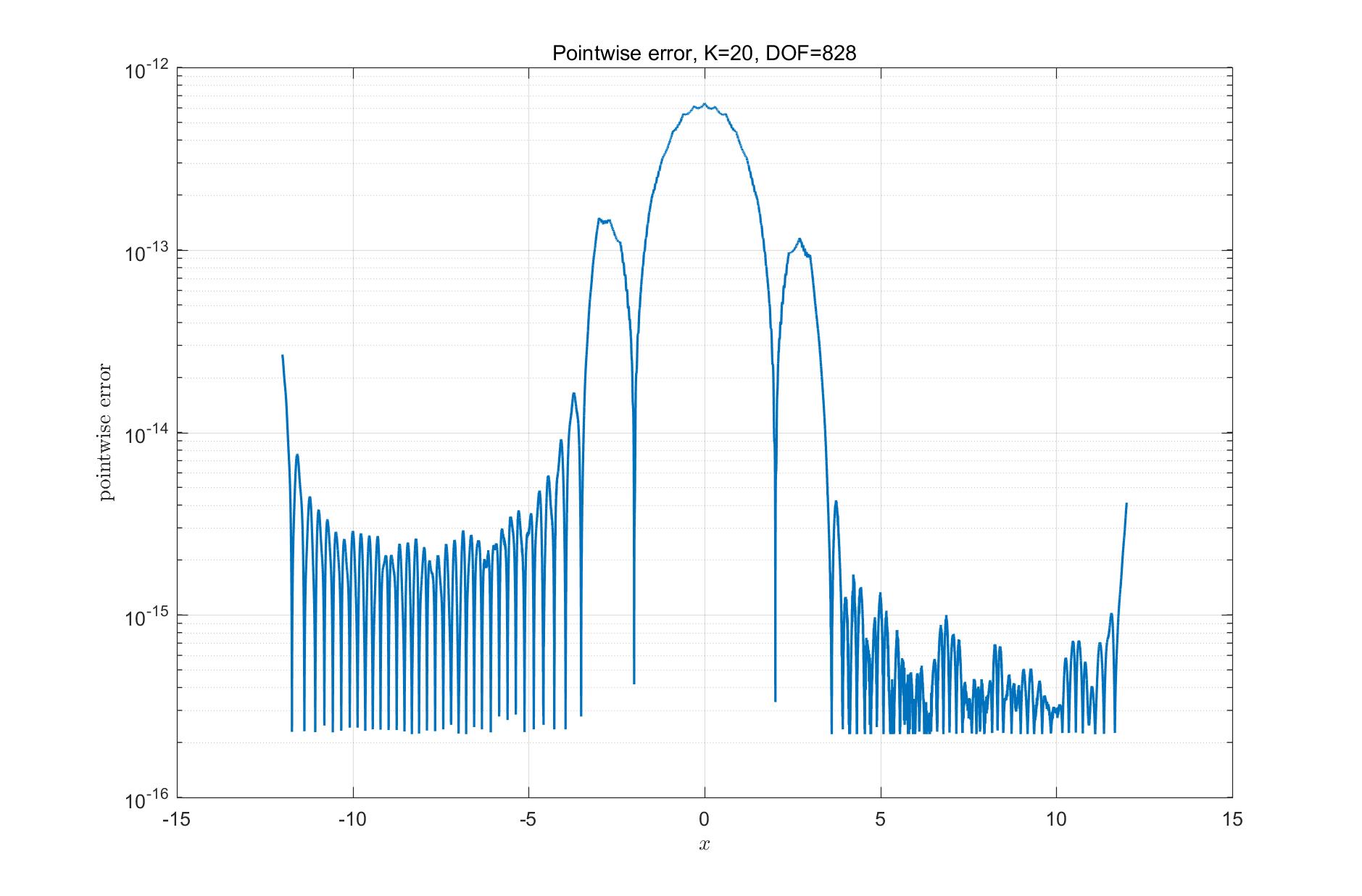}
\caption{Pointwise solution error for the decaying model problem with
$K=20$.}
\label{fig:model_pointwise}
\end{figure}

\section{Conclusion and remarks}
\label{sec:conclusion}

We have proposed a localized core--tail Fourier--Laguerre frame method for
approximation and computation on unbounded domains. The method avoids using a
single global basis for the whole real line. Instead, the finite core and the
far-field tails are represented by basis systems adapted to their different
features: local Fourier extension in the core and modulated Laguerre frames in
the tails.

The main advantage of the construction is the separation of numerical tasks.
In the core, localization reduces the effective frequency scale and allows
oscillatory or locally complicated structures to be resolved by small local FE
systems. The coefficient-energy information of LFE also provides a practical
edge detector for continuous piecewise smooth functions, so that internal
derivative jumps can be aligned with the local partition. In the tails,
modulation removes the dominant oscillatory carrier and leaves only a slowly
varying decaying envelope for the Laguerre expansion.

The error estimates clarify the role of each component. The core error is
governed by local smoothness and local frequency-length products. The tail
error is governed by the Laguerre approximability of the residual envelopes.
Frequency detection does not need to recover exact carrier frequencies; it is
sufficient that the residual phase after modulation remains moderate. The
finite-tail truncation and TSVD regularization errors enter as separate and
controllable terms.

The numerical experiments support these conclusions. The reduced LFE frequency
detector reliably identifies effective modulation centers. Compared with a
fixed-scaling Hermite approximation, the proposed method is particularly
advantageous for oscillatory tails, multi-frequency tails, localized core
structures, and derivative-discontinuous core functions. The decaying model
problem further shows that the representation can be incorporated into a
differential operator discretization by applying the operator analytically to
the basis functions and enforcing interface continuity constraints.

Several extensions remain for future work. A fully automatic adaptive strategy
could combine residual refinement, frequency detection, and internal edge
alignment in a unified algorithm. Higher-dimensional versions, for example via
tensor-product or alternating-direction constructions, would be useful for PDEs
on multidimensional unbounded domains. Fast implementations of local Fourier
extension and reduced-order tail solvers may further improve the efficiency of
the method for large-scale scientific computing problems.

\end{document}